\documentclass[12pt,leqno,a4paper] {amsart}
\usepackage{amssymb,enumerate}
\newcommand{\ZZ}{{\mathbb{Z}}}

\newcommand{\fS} {\mathfrak S}

\newcommand{\Aut}{{{\operatorname{Aut}}}}
\newcommand{\End}{{{\operatorname{End}}}}
\newcommand{\Ext}{{{\operatorname{Ext}}}}
\newcommand{\Hom}{{{\operatorname{Hom}}}}
\newcommand{\rnk}{{{\operatorname{rnk}}}}
\newcommand{\GL}{\operatorname{GL}}
\newcommand{\SL}{\operatorname{SL}}
\newcommand{\Sp}{\operatorname{Sp}}
\newcommand{\GO}{\operatorname{GO}}
\newcommand{\rad}{\operatorname{rad}}
\newcommand{\SO}{\operatorname{SO}}
\newcommand{\Soc}{\operatorname{Soc}}
\newcommand{\tV}{{\tilde V}}
\newcommand\oi{{\bar i}}

\let\al=\alpha
\let\vhi=\varphi
\let\la=\lambda
\let\lra\longrightarrow

\newtheorem{thm}{Theorem}[section]
\newtheorem{lem}[thm]{Lemma}
\newtheorem{cor}[thm]{Corollary}
\newtheorem{prop}[thm]{Proposition}

\newtheorem{thmA}{Theorem}

\theoremstyle{definition}
\newtheorem{rem}[thm]{Remark}

\newtheorem{exmp}[thm]{Example}

\begin{document}

\title[Regular unipotent elements]{Overgroups of regular unipotent\\ 
       elements in simple algebraic groups}

\date{\today}
\author{Gunter Malle}
\address{FB Mathematik, TU Kaiserslautern, Postfach 3049,
  67653 Kaisers\-lautern, Germany.}
\email{malle@mathematik.uni-kl.de}
\author{Donna M. Testerman}
\address{Institut de Math\'ematiques, Station 8, \'Ecole Polytechnique
  F\'ed\'erale de Lausanne, CH-1015 Lausanne, Switzerland.}
\email{donna.testerman@epfl.ch}


\thanks{Work on this article was begun while the authors were visiting the
Mathematical Sciences Research Institute in Berkeley, California in Spring
2018 for the programme ``Group Representation Theory and Applications''
supported by the National Science Foundation under Grant No. DMS-1440140.
Testerman was supported by the Fonds National Suisse de la Recherche
Scientifique grant number 200021-175571. We thank the Isaac Newton Institute
for the Mathematical Sciences, where this work was completed, for support and
hospitality during the programme ``Groups, Representations and Applications:
New Perspectives''. This work was supported by: EPSRC grant number
EP/R014604/1.}


\begin{abstract}
We investigate positive-dimensional closed reductive subgroups of almost simple
algebraic groups containing a regular unipotent element. Our main result states
that such subgroups do not lie inside proper parabolic subgroups unless
possibly when their connected component is a torus.
This extends the earlier result of Testerman and Zalesski treating connected
reductive subgroups.
\end{abstract}

\maketitle


\section{Introduction}

Let $G$ be a simple linear algebraic group defined over an algebraically closed
field. The regular unipotent elements of $G$ are those whose centraliser has
minimal possible dimension (the rank of $G$) and these form a single conjugacy
class which is dense in the variety of unipotent elements of $G$. The main
result of our paper is a contribution to the study of positive-dimensional
subgroups of $G$ which meet the class of regular unipotent elements. Since any
parabolic subgroup must contain representatives from every unipotent conjugacy
class, the question arises only for reductive, not necessarily connected
subgroups, where we establish the following:

\begin{thmA}   \label{thm:main}
 Let $G$ be a simple linear algebraic group over an algebraically closed field,
 $X\le G$ a closed reductive subgroup containing a regular unipotent element
 of~$G$. If $[X^\circ,X^\circ]\ne 1$, then $X$ lies in no proper parabolic
 subgroup of $G$.
\end{thmA}

\noindent
In addition, we show that for many simple groups $G$, there exists a closed
reductive subgroup $X\leq G$ with $X^\circ\ne1$ a torus and such that $X$
meets the class of regular unipotent elements of $G$.
(See Proposition~\ref{prop:torus SL II} and Examples~\ref{exmp:torus orth},
\ref{exmp:torus exc}.) Finally, we go on to consider subgroups of non-simple
almost simple algebraic groups $G$ where there is a well-defined notion of
regular unipotent elements in unipotent cosets of $G^\circ$. We establish the
corresponding result in this setting; see Corollary~\ref{cor:dis}.
\medskip

The investigation of the possible overgroups of regular unipotent elements in
simple linear algebraic groups has a long history.  The \emph{maximal} closed
positive-dimensional reductive subgroups of $G$ which meet the class of regular
unipotent elements were classified by Saxl and Seitz \cite{SS97} in 1997. In
earlier work, see \cite[Thm. 1.9]{Sup95}, Suprunenko obtained a particular case
of their result. In order to derive from the Saxl--Seitz classification an
inductive description of all closed positive-dimensional reductive subgroups
$X\leq G$ containing regular unipotent elements, one needs to exclude that any
of these can lie in proper parabolic subgroups. For connected $X$ this was shown
by Testerman and Zalesski in \cite[Thm. 1.2]{TZ13} in 2013. They then went on to
determine all connected reductive subgroups of simple algebraic groups which
meet the class of regular unipotent elements.
Our result generalises \cite[Thm 1.2]{TZ13} to the disconnected case and thus
makes the inductive approach possible. It is worth pointing out that the
analogous result is no longer true even for simple subgroups once one relaxes
the condition of positive-dimensionality. For example, there exist reducible
indecomposable representations of the group ${\rm PSL}_2(p)$ whose image in the
corresponding $\SL(V)$ contains a matrix with a single Jordan block, i.e., the
image meets the class of regular unipotent elements in $\SL(V)$. In \cite{BT},
Burness and Testerman consider ${\rm PSL}_2(p)$-subgroups of exceptional type
simple algebraic groups which meet the class of regular unipotent elements and
show that with the exception of two precise configurations, such a subgroup does
not lie in a proper parabolic subgroup of $G$ (see \cite[Thms.~1 and~2]{BT}.
\medskip

Our proof of Theorem~\ref{thm:main} relies on the result of Testerman--Zalesski
\cite{TZ13} in the connected case, which actually implies our theorem in
characteristic~0 (see Remark~\ref{rem:char 0}) as well as on results of
Saxl--Seitz \cite{SS97} classifying almost simple irreducible and tensor
indecomposable subgroups of classical groups containing regular unipotent
elements and maximal reductive subgroups in exceptional groups with this
property. For the exceptional groups we also use information on centralisers of
unipotent elements and detailed knowledge of Jordan block sizes of unipotent
elements acting on small modules, as found in Lawther \cite{La95}. For
establishing the existence of positive-dimensional reductive subgroups $X\leq G$,
with $X^\circ$ a torus, and $X$ meeting the class of regular unipotent elements,
we produce subgroups which centralise a non-trivial unipotent element and hence
necessarily lie in a proper parabolic subgroup of $G$.
(See \cite[Thm.~17.10, Cor.~17.15]{MT}.)
\medskip

After collecting some useful preliminary results we deal with the case of
$G=\SL(V)$ in Section~\ref{sec:SL(V)}, with the orthogonal case in
Section~\ref{sec:SO(V)}, and with the simple groups of exceptional type in
Section~\ref{sec:exc}. The case of almost simple groups is deduced from the
connected case in Corollary~\ref{cor:dis}. Finally, in Section~\ref{sec:tori}
we discuss the case when $X^\circ$ is a torus.
\medskip

\noindent
{\bf Acknowledgements}: We thank Steve Donkin, Jacques Th\'evenaz and Adam
Thomas for helpful conversations on cohomology, and Mikko Korhonen and Adam
Thomas for their careful reading of and comments on an earlier version.

\section{Preliminary results}

In this paper we consider almost simple algebraic groups defined over an
algebraically closed field $k$ of characteristic $p\ge0$ and investigate closed
positive-dimensional subgroups which contain a regular unipotent element. For
us, throughout ``algebraic group'' will mean ``linear algebraic group'', and all
vector spaces will be finite-dimensional vector spaces over $k$. An algebraic
group $G$ is called an \emph{almost simple algebraic group} if $G^\circ$ is
simple and $G/Z(G^\circ)$ embeds into $\Aut(G^\circ)$. Thus, $G$ is an extension
of $G^\circ$ by a subgroup of its group of graph automorphisms (see, e.g.,
\cite[Thm.~11.11]{MT}).
As a matter of convention, a ``reductive subgroup'' of an algebraic group will
always mean a closed subgroup whose unipotent radical is trivial. In particular,
a reductive group may be disconnected. For an algebraic group $H$, we write
$R_u(H)$ to denote the unipotent radical of~$H$. Throughout, all $kG$-modules
are rational, as are all extensions, and cohomology groups are those associated
to rational cocycles.

Let us point out that for the question treated here, the precise isogeny type of
the ambient simple algebraic group $G^\circ$ will not matter, as isogenies
preserve parabolic subgroups and induce isomorphisms on the unipotent variety
and so also preserve regular unipotent elements. (If $G$ is almost simple and
$p$ does not divide the order of the fundamental group of $G^\circ$, the
natural map $G\to G/Z(G^\circ)$ induces an isogeny of $G^\circ$ onto its
adjoint quotient, preserving regular unipotent elements in $G$; in the general
case, a reduction to $G^\circ$ of adjoint type is given in \cite[I.1.7]{Sp82}.)
In particular, for $G$ a
classical type simple algebraic group we will argue for the groups $\SL(V)$,
$\Sp(V)$ and $\SO(V)$, and for the groups of type $B_l$ and $C_l$ defined over
$k$ of characteristic $2$, we may choose to work with whichever group is more
convenient under the given circumstances.

We start by making two useful observations which will simplify the later
analysis.

\begin{rem}   \label{rem:char 0}
In the situation of Theorem~\ref{thm:main}, assume that $p=0$. As $|X:X^\circ|$
is finite, some power of a regular unipotent element $u\in X$ will lie in
$X^\circ$. In characteristic~0 any power of a regular unipotent element is
again regular unipotent, so here we are thus reduced to studying the
\emph{connected} reductive subgroup $X^\circ$ satisfying the same assumptions.
In that case, the conclusion of Theorem~\ref{thm:main} was established in
\cite[Thm.~1.2]{TZ13}. Hence, in proving Theorem~\ref{thm:main} we may assume
$p>0$ whenever convenient. Furthermore, we will assume without loss of
generality that $X=X^\circ\langle u\rangle$.
\end{rem}

\begin{rem}   \label{rem:transitive}
 Let $X=X^\circ\langle u\rangle$ be a reductive subgroup of a connected
 reductive group $G$ such that $u$ is regular unipotent in $G$ and
 $[X^\circ,X^\circ]\ne1$. Let $X_1$ be one of the simple components of
 $[X^\circ,X^\circ]$ and set $H:=\langle X_1,u\rangle$. Then
 $H^\circ=\prod_i X_1^{u^i}$, so $\langle u\rangle$
 acts transitively on the set of simple components of $H^\circ$, and if $X$
 lies in a proper parabolic subgroup of $G$, then so does $H$. Thus,
 when proving Theorem~\ref{thm:main} we may as well assume that the simple
 components of $X^\circ$ are permuted transitively by $u$.
\end{rem}

\subsection{Jordan forms and tensor products}

The following elementary fact will be used throughout (see also
\cite[Lemma~1.3(i)]{SS97}):

\begin{lem}   \label{lem:power}
 Assume that $p>0$ and let $u\in\SL(V)$ be unipotent with a single Jordan block.
 Write $\dim V=ap+b$ with $0\le b<p$. Then $u^p$ has $p$ Jordan blocks, $b$ of
 size~$a+1$ and the other $p-b$ of size~$a$.
\end{lem}

\begin{lem}   \label{lem:tens 1}
 Let  $u\in\SL(V)$ be a unipotent element with a single Jordan block of
 size~$n=\dim V$, or with two Jordan blocks of sizes $n-1,1$ or  $n-2,2$. If $u$
 preserves a non-trivial tensor product decomposition of $V$ then $\dim V=4$ and
 $u$ has two Jordan blocks on $V$. If $p=2$ these are of sizes $2,2$.
\end{lem}

\begin{proof}
Using the description of Jordan block sizes of unipotent elements in tensor
products given in \cite[Lemma~1.5]{SS97} we see that necessarily $\dim V=4$ and
either $u$ has Jordan block sizes $2,2$, or $p\ne2$ and $u$ has Jordan block
sizes $3,1$, as claimed.
\end{proof}

Before establishing a useful consequence of Lemma~\ref{lem:tens 1}, we recall
the following well-known Clifford theoretic result, see, e.g.,
\cite[Prop.~2.6.2]{BGMT}:

\begin{lem}   \label{lem:cyc ext}
 Let $N\unlhd H$ be groups with $H/N$ finite cyclic and $V$ be a
 finite-dimensional irreducible $kH$-module. Then $V|_N=\bigoplus_g U^g$, where
 $U$ is any irreducible $kN$-submodule of $V$ and $g$ runs over a system of
 coset representatives of the stabiliser of $U$ in $H$. Moreover, the $U^g$ are
 pairwise non-isomorphic $kN$-modules.
\end{lem}

We now show the desired corollary of Lemma~\ref{lem:tens 1}:

\begin{lem}   \label{lem:tensor}
 Let $H\le\SL(V)$ be connected reductive with non-trivial derived subgroup and
 assume that $V|_H$ is completely reducible and homogeneous. If $H$ is
 normalised by a unipotent element $u\in\SL(V)$ with a single Jordan block of
 size~$n=\dim V$, or with two Jordan blocks of sizes $n-1,1$ or $n-2,2$ then
 either $V$ is an irreducible $H$-module, or $\dim V=4$, $H\langle u\rangle$
 preserves a non-trivial tensor product decomposition of $V$, and $u$ has two
 Jordan blocks on $V$, of sizes $2,2$ if $p=2$.
\end{lem}

\begin{proof}
Let $u\in\SL(V)$ be the unipotent element normalising $H$ as in the assumption.
Let $V_1\le V$ be an irreducible $H\langle u\rangle$-submodule of $V$. As $V$ is
homogeneous as an $H$-module, Lemma~\ref{lem:cyc ext} shows that $V_1|_H$ is
irreducible. Since $[H,H]\ne1$, we have $\dim V_1>1$. Then with
$W=\Hom_H(V_1,V)$ we have $V\cong V_1\otimes W$ as an $H$-module, and this
decomposition is stabilised by $u$ (see e.g.~\cite[Prop.~18.1]{MT}). Applying
Lemma~\ref{lem:tens 1}, this implies that either $\dim W=1$, whence $V=V_1$ is
irreducible for $H$, or we are in the exceptional case of that result, as in the
conclusion.
\end{proof}

The proof of the next result is modelled after the proof of
\cite[Prop.~2.1]{SS97} which treats a more special situation:

\begin{lem}   \label{lem:tens 2}
 Assume $p>0$ and let $X=X^\circ\langle u\rangle\le\SL(V)$ be a reductive
 subgroup and $u\in\SL(V)$ a unipotent element with a single Jordan block of
 size~$n:=\dim V>1$, or with two Jordan blocks of sizes $n-1,1$ or $n-2,2$. If
 $X^\circ$ acts irreducibly on $V$, then either $[X^\circ,X^\circ]$ is simple,
 or one of the following holds:
 \begin{enumerate}[\rm(1)]
  \item $\dim V=4$, $X$ preserves a non-trivial tensor decomposition of~$V$ and
   $u$ has two Jordan blocks on $V$. If $p=2$ these are of size $2,2$;
  \item $p\in\{2,3\}$, $V=V_1\otimes\cdots\otimes V_p$ as an $X^\circ$-module
   with $\dim V_i=2$, $[X^\circ,X^\circ]=A_1^p$, $u$ permutes both sets of
   factors transitively and has a single Jordan block on $V$. Moreover, $u^p$
   has a single Jordan block on each $V_i$; or
  \item $p=2$, $V=V_1\otimes V_2$ as an $X^\circ$-module with $\dim V_i=3$, $u$
   has Jordan blocks of sizes $8,1$ on $V$ and $u^2$ has a single Jordan block
   on each $V_i$.
 \end{enumerate}
 Here, in $(2)$ and $(3)$, $X$ does not preserve the stated tensor product
 decomposition of~$V$.
\end{lem}

\begin{proof}
Note that $[X^\circ,X^\circ]\ne1$ as $\dim V>1$ and $X^\circ$ acts irreducibly.
Write $[X^\circ,X^\circ]=X_1\cdots X_s$ with simple algebraic groups $X_i$, so
$V=V_1\otimes\cdots\otimes V_s$ with non-trivial irreducible $X_i$-modules
$V_i$. Now $u$ permutes the factors $X_i$ and their corresponding tensor factors
$V_i$. Assume that $s>1$. If $u$ has at least two orbits on the set of $X_i$,
this yields a corresponding $u$-invariant tensor decomposition of $V$. By
Lemma~\ref{lem:tens 1} we reach case~(1). 
\par
Henceforth, we may assume that $u$ permutes the $X_i$, and thus the $V_i$,
transitively. In particular all $V_i$ have the same dimension~$m$, that is,
$\dim V=m^s$, and $s=:p^a>1$ is a power of $p$. Let $b$ be minimal with
$p^b\ge m$. Now, $u^{p^a}$ stabilises all $V_i$, so is a tensor product of
matrices of size $m$ and thus of order at most $p^b$. Hence $|u|$ divides
$p^{a+b}$. On the other hand,
$$p^{a+b}\ge |u|\ge\dim V-2=m^s-2= m^{p^a}-2 > p^{(b-1)p^a}-2.$$
The above conditions imply that either $a=b=1$, $m=2$ and $s=p\le3$, or $a=1$,
$b=s=p=2$ and $m=3$.   \par
In the first case, $\dim V=m^s=2^p$, $p\le3$, and as $m=2$ all simple factors
$X_i$ of $[X^\circ,X^\circ]$ must have type $A_1$, as in (2). The statement
about the Jordan form of $u$ follows from Lemma~\ref{lem:power}. \par 
In the second case we have $\dim V=9$, and our inequalities force that $|u|=8$
and hence $u$ has Jordan blocks of sizes $8,1$ or $7,2$ and by
Lemma~\ref{lem:power}, $u^2$ has Jordan blocks of sizes~$4,4,1$,
respectively~$4,3,1,1$. The latter cannot arise as the block sizes of a tensor
product of two $3\times3$ unipotent matrices by \cite[Lemma~1.5]{SS97}, so we
are in the former case and $u^2$ has a single Jordan block on each $V_i$, as
in~(3).
\end{proof}

\subsection{On subgroups containing regular unipotent elements}

For connected groups, the following result from \cite[Lemma~2.6]{TZ13} will be
useful:

\begin{lem}   \label{lem:reg in Levi}
 Let $G$ be connected reductive, $P\le G$ a parabolic subgroup with Levi
 complement $L$ and assume that $u\in P$ is regular unipotent in $G$. Then the
 image of  $u$ is regular unipotent in $L$ and hence in each simple factor of
 $[L,L]$.
\end{lem}

\begin{lem}   \label{lem:reg inner}
 Let $G$ be simple and $H\le G$ be a connected reductive subgroup normalised by
 a regular unipotent element $u$ of $G$. Assume that $H=Y_1\circ Y_2$ is a
 central product with $Y_1\ne1$ such that $u$ acts by an inner automorphism on
 $Y_1$. Then $Y_1=H$ contains a regular unipotent element of $G$.
\end{lem}

\begin{proof}
By assumption, $u$ acts as an inner automorphism on $Y_1$, say by $z\in Y_1$.
Thus, $uz^{-1}$ and $Y_2$ are contained in $C_H(Y_1)$ and so
$u\in H\langle u\rangle=(Y_1Y_2)\langle u\rangle
 \leq Y_1\circ C_H(Y_1)$. But then \cite[Prop.~2.3]{TZ13}
implies that $Y_2=1$. Now $u=z\cdot uz^{-1}$ is regular unipotent. Replacing
$uz^{-1}$ and $z$ by their unipotent parts respectively, we may assume both to
be unipotent and lying in a common Borel subgroup of $G$. As $uz^{-1}$
centralises $Y_1$ and thus isn't regular, $z\in Y_1$ must be regular by
\cite[Lemma~2.4]{TZ13}.
\end{proof}

\begin{lem}   \label{lem:Iulian}
 Let $G$ be simple and $H<G$ a connected reductive subgroup containing a
 regular unipotent element $u$ of $G$. Then $u$ is regular in $H$.
\end{lem}

\begin{proof}
Let $B<H$ be a Borel subgroup of $H$ containing $u$. Assume $u$ is not regular
unipotent in $H$. By \cite[Ch.~III, 1.13]{SpSt} it may be written as a
product of root elements
$$u=\prod_{\al}u_\al(1)\prod_{\beta}u_\beta(c_\beta)\quad\text{for
  suitable $c_\beta\in k$},$$
where the first product runs over a proper subset of the simple roots of the
root system $\Phi$ of $H$ with respect to the pair $(T,B)$ where $T<B$ is some
maximal torus, and the second one over the roots in $\Phi^+$ of height at
least~2. Thus, $u$ lies in the unipotent radical of the parabolic subgroup of
$H$ whose Levi factor is generated by the root subgroups for the simple roots
not occurring in the representation of $u$ and their negatives, which thus is
not a torus.   \par
By Borel--Tits,
$u$ then also lies in the unipotent radical of a proper parabolic subgroup $P$
of $G$ with non-toral Levi factor. But then, when writing $u$ as a product of
root elements for $G$ with respect to a Borel subgroup contained in $P$, not
all simple roots can occur, whence $u$ is not regular in $G$. This
contradiction achieves the proof.
\end{proof}

We will make frequent use of the following result, the second part of which was
essentially shown by Saxl and Seitz \cite[Prop.~2.2]{SS97}:

\begin{prop}   \label{prop:new SS}
 Let $X=X^\circ\langle u\rangle$ be a reductive subgroup of the simple
 classical group $G=\SL(V)$, $\Sp(V)$ or $\SO(V)$ with $X^\circ$ simple and
 irreducible on $V$, where $\dim V\ge7$ when $G$ is of orthogonal type. If $X$
 contains a regular unipotent element $u$ of $G$, then $X^\circ$ acts tensor
 indecomposably on $V$.   \par
 Furthermore, either $X=B_3<G=D_4$, or $X$ and the highest weight of $X^\circ$
 on $V$ are as in Table~\ref{tab:SS97} (up to Frobenius twists and taking
 duals) and $u$ has a single Jordan block.
\end{prop}

\begin{table}[htb]
\caption{Simple modules with regular unipotent elements}   \label{tab:SS97}
$$\begin{array}{c|ccccc}
 X& \la& \dim V& \text{cond.}& |u|\cr
\hline
 A_1& m\varpi_1& m+1& m<p& p\cr
 A_l& \varpi_1& l+1& l>1& <p(l+1)\cr
 B_l& \varpi_1& 2l+1& p>2& <p(2l+1)\cr
 C_l& \varpi_1& 2l& & <2pl\cr
 G_2& \varpi_1& 7& p>2& \le p^2\cr
 G_2& \varpi_1& 6& p=2& 8\cr
 A_2.2& \varpi_1+\varpi_2& 8& p=2& 8\cr
 D_l.2& \varpi_1& 2l& p=2,l\ge3& <4l\cr
\end{array}$$
The last column records (an upper bound for) \\
the order of a regular unipotent element $u\in X$.
\end{table}

\begin{proof}
Assume that $V=V_1\otimes V_2$ for non-trivial irreducible $X^\circ$-modules
$V_i$. If $X=X^\circ$, then Lemma~\ref{lem:tens 1} gives that $\dim V=4$ and
$G=D_2$, but this is not simple. If some power $u^j$ acts as an inner element
$y$ on $X^\circ$, then $u^jy^{-1}=z\in X$ centralises $X^\circ$, hence, as
$X^\circ$ is irreducible, we must have that $u^j$ equals the unipotent part of
$y$ and so lies in $X^\circ$.
So now we may assume that $u\notin X^\circ$ and hence $X=A_l.2$ ($l\geq 2$),
$D_l.2$ ($l\geq 4$) or $E_6.2$, with $p=2$, or $X=D_4.3$ and $p=3$. 
Unipotent elements in $X=E_6.2$ have order at most~32, but there is no faithful
irreducible representation of $X$ of dimension at most 34. Thus $X^\circ$ is
of classical type. Let $d$ denote the dimension of its natural module. Then,
e.g., by \cite[Tab.~2]{Lue} we have $\dim V_i\ge d$, so $\dim V\ge d^2$, but
unipotent elements of $X$ have order less than $p^2d$. Thus $d< p^2$. Note
that $D_4$ with $p=3$ cannot occur, as here unipotent elements have order at
most~$27<d^2-2=62$. This only leaves the possibility $X^\circ=A_2$, $p=2$, and
$\dim V=9$.  But there is no 9-dimensional irreducible orthogonal module in
characteristic~2, so in fact $V$ must be tensor indecomposable for $X^\circ$.
The remaining assertions are now shown in \cite[Prop.~2.2]{SS97}.
\end{proof}

We will obtain a similar classification for unipotent elements in $\SL(V)$ with
a Jordan block of size~$\dim V-1$ when $p=2$ in Proposition~\ref{prop:l-1}.

\subsection{Jordan forms and orders of regular unipotent elements}

While the notion of regular unipotent element is well-known for connected
reductive groups, this is much less so for non-connected reductive groups.
Still, similar results hold.

Let $G$ be a not necessarily connected reductive algebraic group and let
$x\in G$ be unipotent.
Spaltenstein \cite[p.~41 and II.10.1]{Sp82} has shown (generalising a result
of Steinberg in the connected case) that the coset $xG^\circ$ of the connected
component $G^\circ$ contains a unipotent $G^\circ$-conjugacy class $C$ that is
dense in the variety of unipotent elements of $xG^\circ$, called the class of
\emph{regular unipotent elements of $xG^\circ$}. Since this variety is
irreducible \cite[Cor.~I.1.6]{Sp82}, $C$ is also the unique class of unipotent
elements in $xG^\circ$ of maximal dimension.

We now describe the Jordan block structure of regular unipotent elements in
classical type almost simple groups on their natural representation.
For us, the \emph{natural representation} for the extension $A_{l-1}.2$ of
$A_{l-1}$, $l\ge3$, by its graph automorphism of order~2 is defined by its
embedding into the stabiliser in $\GO_{2l}$ of a pair of complementary totally
singular subspaces, with $A_{l-1}$ acting in its natural representation,
respectively its dual, on these subspaces. We do not consider
$D_4.3$ or $D_4.\fS_3$ as being of classical type.

\begin{lem}   \label{lem:Jordan}
 Let $G=G^\circ\langle u\rangle$ be almost simple of classical type with $u$
 regular unipotent in $uG^\circ$. Then in the natural representation of $G$:
 \begin{enumerate}[\rm(a)]
  \item $u$ has a single Jordan block for $G=A_l$, $B_l$
   (when $p\ne2$), $C_l$, and for $G=D_l.2$ when $p=2$;
  \item $u$ has two Jordan blocks of sizes $2l,1$ when $p=2$ for $G=B_l$;
  \item $u$ has two Jordan blocks of sizes $2l-1,1$ when $p\ne2$, respectively
   of sizes $2l-2,2$ when $p=2$ for $G=D_l$;
  \item $u$ has a single Jordan block of size $2l$ when $p=2$ for $G=A_{l-1}.2$
   with $l$ odd; and
  \item $u$ has two Jordan blocks of sizes $2l-2,2$ when $p=2$ for
   $G=A_{l-1}.2$ with $l$ even.
 \end{enumerate}
\end{lem}

\begin{proof}
Only~(d) and~(e) are not shown in \cite[Lemma~1.2]{SS97}. Spaltenstein
\cite[I.2.7, I.2.8(c)]{Sp82} gives a description of the unipotent classes in 
$A_{l-1}.2\setminus A_{l-1}$ in terms of the Jordan normal form of the square
of the elements on the natural $A_{l-1}$-module, and a formula for the
centraliser dimension. From this it can be seen that elements $u$ with minimal
centraliser dimension are those for which $u^2$ has one Jordan block of size
$l$ if $l$ is odd, and two blocks of sizes $l-1,1$ if $l$ is even. Thus, in the
natural $2l$-dimensional orthogonal representation of $A_{l-1}.2$, the element $u^2$ has two
Jordan blocks of size $l$, respectively four of sizes $l-1,l-1,1,1$. Given the
possible Jordan block shapes of unipotent elements in $D_l.2$ in its natural  
representation, the claim for $u$ follows with Lemma~\ref{lem:power}.
\end{proof}

Since the natural representations are faithful, the above result also allows one
to read off the orders of regular unipotent elements of almost simple classical
groups.

\subsection{Some results on extensions}
We conclude this preparatory section by collecting some basic properties on
Ext-groups. We state the following well-known result for future reference
(see \cite[Prop.~3.3.4]{Wei}).

\begin{lem}   \label{lem:(*)}
 Let $H$ be a group, and $V_i,U_j$ with $1\le i\le n$, $1\le j\le m$, be
 $kH$-modules. Then
 $$\Ext^1_{H}(\bigoplus_i V_i,\bigoplus_j U_j)
   \cong\bigoplus_{i,j}\Ext^1_{H}(V_i,U_j).$$
\end{lem}

We thank Jacques Th\'evenaz for pointing out the following result:

\begin{lem}   \label{lem:ext fact}
 Let $k$ be a field, $N\unlhd H$ be groups and $U,V$ two finite-dimensional
 $kH$-modules on which $N$ acts trivially. Assume that $\Ext_N^1(k,k)=0$.
 Then $\Ext^1_{H/N}(U,V)\cong\Ext^1_H(U,V)$.
\end{lem}
 
\begin{proof}
We use the (exact) inflation-restriction sequence for cohomology (see
\cite[6.8.3]{Wei}) for a $kH$-module $M$:
$$0\rightarrow H^1(H/N,M^N)\buildrel\text{inf}\over\longrightarrow H^1(H,M)
  \buildrel\text{res}\over\longrightarrow H^1(N,M)^{H/N}\rightarrow\ldots$$
which by \cite[6.1.2]{Wei} can be interpreted as the $\Ext$-sequence
$$0\rightarrow\Ext_{H/N}^1(k,M^N)\rightarrow\Ext_H^1(k,M)\rightarrow
  \Ext_N^1(k,M)^{H/N}\rightarrow\ldots.$$
Applying this with $M=U^*\otimes V$ and using
$\Ext^1(k,U^*\otimes V)\cong\Ext^1(U,V)$ (see \cite[Cor.~1]{AR67}) the previous
sequence becomes
$$0\rightarrow\Ext_{H/N}^1(k,(U^*\otimes V)^N)\rightarrow\Ext_H^1(U,V)
  \rightarrow\Ext_N^1(U,V)^{H/N}\rightarrow\ldots.$$
As $N$ acts trivially on $U$ and $V$, the first term equals
$\Ext_{H/N}^1(k,U^*\otimes V)\cong\Ext_{H/N}^1(U,V)$, while the third is
$\Ext_N^1(k^{\dim U},k^{\dim V})^{H/N}=0$ by our hypothesis on $\Ext_N^1(k,k)$
and Lemma~\ref{lem:(*)}, whence exactness of the sequence implies our claim.
\end{proof}

\begin{lem}   \label{lem:ext}
 \begin{enumerate}
  \item[\rm(a)] Let $H$ be a semisimple algebraic group. Then there are no
   non-trivial self-extensions between irreducible $H$-modules.
  \item[\rm(b)] Let $H_1\circ H_2$ be a semisimple group acting on
   $V_1\oplus V_2$, with $H_i$ acting trivially on $V_{3-i}$. Then
   $$H^1(H_1H_2,V_1\oplus V_2) \cong H^1(H_1,V_1)\oplus H^1(H_2,V_2).$$
 \end{enumerate}
\end{lem}

\begin{proof}
Part~(a) is \cite[II.2.12(1)]{Ja03}.
In (b) we may write $H_1\circ H_2=(H_1\times H_2)/Z$ for a finite central
subgroup $Z\le H_1\times H_2$ acting trivially on $V_1\oplus V_2$, with
$\Ext_Z^1(k,k)=0$. Then by Lemmas~\ref{lem:ext fact} and~\ref{lem:(*)} we have
$$H^1(H_1 H_2, V_1\oplus V_2)\cong H^1(H_1\times H_2, V_1\oplus V_2)
  \cong H^1(H_1\times H_2,V_1)\oplus H^1(H_1\times H_2,V_2).$$
Since $H_i$ acts trivially on $V_{3-i}$, applying Lemma~\ref{lem:ext fact}
again we obtain $H^1(H_1\times H_2,V_i)\cong H^1(H_i,V_i)$ for $i=1,2$, as
claimed.
\end{proof}

\section{The case of $\SL(V)$}   \label{sec:SL(V)}

In this section we prove Theorem~\ref{thm:main} for those classical type simple
algebraic groups for which regular unipotent elements have a single Jordan
block on their natural module (see Lemma~\ref{lem:Jordan}). We are in the
following situation:
$X\le\SL(V)$ is a (not necessarily connected) reductive subgroup of the form
$X=X^\circ\langle u\rangle$ for a regular unipotent element $u$ of $\SL(V)$. We
also assume that $X^\circ$ is not a torus; this case will be considered
in Section~\ref{subsec:tori SL}. We will show that $X$ cannot be contained
in  a proper parabolic subgroup of $\SL(V)$, that is, $X$ acts irreducibly on
$V$. For this, we may whenever convenient, assume that $X^\circ$ is semisimple,
since if $[X^\circ,X^\circ]\langle u\rangle$ is not contained in a proper
parabolic subgroup of $\SL(V)$ then neither is $X^\circ\langle u\rangle$.
As $u$ has a single Jordan block on $V$, if $V_1<V_2\le V$ are $u$-invariant
subspaces, then $u$ has a single Jordan block on $V_2/V_1$.

\subsection{The completely reducible case}
We first deal with the case when $V$ is completely reducible for $X^\circ$.

\begin{lem}   \label{lem:cr=irr}
 Let $X\le \SL(V)$ be a reductive subgroup of the form
 $X=X^\circ\langle u\rangle$ for a regular unipotent element $u$ of $\SL(V)$,
 such that $X^\circ$ is not a torus. Assume that $V|_{X^\circ}$ is completely
 reducible. Then $V$ is an irreducible $X$-module.
\end{lem}

\begin{proof}
Decompose $V|_{X^\circ}=V_1\oplus\cdots\oplus V_m$ into its homogeneous
components $V_i$. Then $\langle u\rangle$ acts on the set of components,
transitively since otherwise we obtain a $u$-invariant decomposition of $V$,
contradicting regularity of $u$. Thus $v:=u^m$ stabilises each component $V_i$,
acting as a single Jordan block on it by Lemma~\ref{lem:power}, and $V_i$ is
homogeneous as an $X^\circ$-module. Since $X^\circ$ is not a torus,
Lemma~\ref{lem:tensor} implies that $V_i$ is an irreducible $X^\circ$-module.
As the $V_i$ are permuted transitively by $u$, the module $V$ is irreducible for~$X$.
\end{proof}

In view of Lemma~\ref{lem:cr=irr}, it seems interesting to determine the
structure of irreducible subgroups containing regular unipotent elements.

\begin{prop}   \label{prop:irr new}
 Assume $p>0$. Let $X\le \SL(V)$ be a reductive subgroup of the form
 $X=X^\circ\langle u\rangle$ for a regular unipotent element $u$ of $\SL(V)$,
 such that $X^\circ\ne1$ is semisimple. Assume that $X$ acts irreducibly on $V$.
 Then
 \begin{enumerate}[$\bullet$]
  \item $V|_{X^\circ}=V_1\oplus\cdots\oplus V_m$ for submodules $V_i$,
   transitively permuted by $u$, on which $u^m$ acts with a single Jordan
   block; and
  \item $X^\circ=Y_1\cdots Y_{r}$ for distinct semisimple normal subgroups
   $Y_1,\ldots,Y_r$ transitively permuted by $u$, with $r|m$, where $Y_i$ acts
   trivially on $V_j$ for $j\not\equiv i\pmod r$, and in addition one of the
   following holds:
   \begin{enumerate}[\rm(1)]
    \item the $Y_i$ are isomorphic simple algebraic groups and $V_i$ is an
     irreducible tensor indecomposable $Y_{i\pmod r}$-module, for $1\le i\le m$;
     if $r<m$ then $p=2$, $Y_i=A_l$ for some $l\ge2$, $\dim V_i=l+1$ and $m=2r$;
     or
    \item $p\in\{2,3\}$, $m=r$, $Y_i=A_1^p$ for all $i$, with $u^m$ permuting
     the $p$ factors transitively, and each $V_i$, as a $Y_i$-module, is a
     tensor product of $p$ irreducible $2$-dimensional $A_1$-modules.
   \end{enumerate}
 \end{enumerate}
\end{prop}

\begin{proof}
As $V$ is irreducible for $X$, by Lemma~\ref{lem:cyc ext},
$V|_{X^\circ}=V_1\oplus\cdots\oplus V_m$ is a direct sum of non-isomorphic
irreducible $X^\circ$-modules permuted transitively by $u$, with
$u(V_i)=V_{i+1}$ for all $i$. Thus $u^m$
stabilises each $V_i$ and acts with a single Jordan block on it by
Lemma~\ref{lem:power}. The semisimple group $X^\circ$ is a product of
simple normal subgroups. Let $Y_1$ denote the product of those simple factors of
$X^\circ$ acting non-trivially on $V_1$. By Lemma~\ref{lem:tens 2} applied to
$Y_1\langle u^m\rangle$ acting on $V_1$, either $Y_1$ is simple, and then it
acts tensor indecomposably on $V_1$ by Proposition~\ref{prop:new SS}; or
$p\in\{2,3\}$, $Y_1=A_1^p$, $u^m$ permutes the $p$ factors transitively, and
$V_1$ is a tensor product of $p$ irreducible 2-dimensional $A_1$-modules.
\par
Now $Y_{i+1}:=Y_1^{u^{-i}}$ acts non-trivially on $u^i(V_1)=V_{i+1}$ for
$i=1,\ldots,m-1$. Thus $X^\circ$ is generated by, and hence equal to the
product of the $Y_i$, which are permuted transitively by $u$. Since each $Y_i$
is an orbit of $u^m$ on the set of simple factors of $X^\circ$, these sets
are mutually disjoint. Let $u^r$ be the smallest
power of $u$ stabilising $Y_1$. Then $X^\circ=Y_1\cdots Y_r$ and $r|m$.   \par
Assume $Y_i$ is simple and $r\ne m$. As $u^r$ normalises $Y_1$ but does not
stabilise
$V_1$, it must act as an outer automorphism of the simple group $Y_1$. So
$p\in\{2,3\}$ and moreover, $u^{rp}$ acts as an inner automorphism on $Y_1$
and stabilises the $Y_1$-module $V_1$. So $u^{rp}\le\langle u^m\rangle$, and as
$r|m$ this implies that $m=rp$. By Proposition~\ref{prop:new SS} the group
$Y_1\langle u^m\rangle$ acting on $V_1$ occurs in Table~\ref{tab:SS97}.
If $p=3$, so $Y_1=D_4$, there is no such entry. So we have $p=2$. The case
$D_l.2$ on the natural module does not occur here, as $V_1$ is not
$u^r$-stable. So $Y_1=A_l$ with $l\ge2$, $Y_1\langle u^r\rangle=A_l.2$ and
$\dim V_1=\dim u^r(V_1)=l+1$ while $u^{2r}=u^m$ stabilises all $V_i$, as in (1).
\par
Finally, assume that $p\in\{2,3\}$ and $Y_1=A_1^p$. Then $Y_1$ is normalised by
$u^r$, and $u^m$ permutes its $p$ factors transitively, where $r|m$. This
forces $r=m$, as in (2).
\end{proof}

\begin{rem}   \label{rem:prop irr}
 In the situation of Proposition~\ref{prop:irr new}, for $1\le i\le r$ set
 $W_i:=V_i$ if $m=r$, and $W_i:=V_i\oplus u^r(V_i)$ if $m=2r$. Then $W_i$ is an
 irreducible
 $Y_i\langle u^r\rangle$-module and $Y_i$ acts trivially on $W_j$ for $j\ne i$.
 So $W_i':=\bigoplus_{j\ne i} W_j$ is the trivial $Y_i$-homogeneous component of
 $V$ and the decomposition $V=W_i\oplus W_i'$ is stabilised by $u^r$. Moreover,
 $u^r$ acts by a single Jordan block on each $W_i$.
\end{rem}

\subsection{The not completely reducible case}

\begin{lem}   \label{lem:soc}
 Let $X\le\SL(V)$ be a reductive subgroup of the form
 $X=X^\circ\langle u\rangle$ for a regular unipotent element $u$ of $\SL(V)$,
 with $[X^\circ,X^\circ]\ne1$.
 Then $\Soc(V|_{[X^\circ,X^\circ]})$ is an irreducible $X$-module. That is,
 the socle series of $V$ as $[X^\circ,X^\circ]$-module is a composition series
 as $X$-module.
\end{lem}

\begin{proof}
By definition, $\Soc(V|_{[X^\circ,X^\circ]})$ is a completely reducible
$[X^\circ,X^\circ]$-module, on which $u$ acts as a single Jordan block since
that property passes to quotients and submodules. So by Lemma~\ref{lem:cr=irr},
it is an irreducible $X$-module. The claim now follows by induction
on the length of the socle series of $V|_{[X^\circ,X^\circ]}$. 
\end{proof}

\begin{lem}   \label{lem:innercent}
 Let $N\unlhd H$ be groups with $H=N\langle v\rangle$. Assume that $v$ has
 order $p^a$ and acts by an inner automorphism on $N$. If $N$ does not contain
 elements of order $p^a$ then $H=N\circ\langle z\rangle$ for some element
 $z\in H$ of infinite order or of order a multiple of $p^a$.
\end{lem}

\begin{proof}
By assumption there is $y\in N$ such that $z:=vy$ centralises $N$. As $z$
centralises $y$, either both $y,z$ are of infinite order, or both have finite
order and then the order $p^a$ of $v=zy^{-1}$ divides the least common multiple
of $|z|,|y|$. In the latter case, $|y|$ is not divisible by $p^a$ by assumption,
which implies that $p^a$ divides $|z|$. Clearly,
$\langle N,z\rangle=\langle N,v\rangle=H$, whence our claim.
\end{proof}

\begin{lem}   \label{lem:unip centr}
 Let $H\le\SL(V)$ be a group, and assume that all composition factors of
 $V|_{H}$ are mutually non-isomorphic. Then $Z(H)$ contains no non-trivial
 unipotent elements.
\end{lem}

\begin{proof}
Let $u\in Z(H)$ be unipotent. Set $x = u-1\in C_{\End(V)}(H)$, so $x$ acts as~0
on any irreducible $kH$-subquotient of $V$. Let $V_1=\ker x$ and set
$\bar V=V/V_1$. If $x\ne0$ then $V_1\ne V$, so $\bar V\ne0$.
Let $\bar U\le \bar V$ be an irreducible $kH$-submodule. Then $x.U\le V_1$
where $U$ is the full preimage of $\bar U$ in $V$. But then
$x.U\cong U/(U\cap\ker x)=U/V_1=\bar U$ is a simple submodule of $V_1$
isomorphic to $\bar U$, contradicting our assumption. Thus $x=0$ and so $u=1$. 
\end{proof}

We now use results of McNinch on semisimplicity of low-dimensional modules in
order to study extensions of irreducible modules for simple algebraic groups
on which a full Jordan block acts.

\begin{prop}   \label{prop:ext full}
 Let $X=X^\circ\langle u\rangle\le\SL(V)$ with $X^\circ$ a simple algebraic
 group and $u$ regular unipotent in $\SL(V)$. If $V$ has at most two
 $X$-composition factors then $X$ acts irreducibly on~$V$.
\end{prop}

\begin{proof}
By the main result of \cite{TZ13} we may assume that $X$ is not connected.
We argue by contradiction and so assume $V$ has two $X$-composition factors
$V_1$, $V_2:=V/V_1$. Then by Lemma~\ref{lem:soc} these are the socle and the
head of $V$ as an $X^\circ$-module. Applying Proposition~\ref{prop:irr new} to
$V_i$ and invoking the hypothesis that $X^\circ$ is simple, we find that
\ref{prop:irr new}(1) holds with $r=1$. In particular, for each $i$ we have
$V_i=V_{i1}\oplus\cdots\oplus V_{im_i}$ with irreducible $X^\circ$-modules
$V_{ij}$ where $m_i=1$, or $p=2$ and $m_i\in\{1,2\}$. Note that for any $j$
there must be a non-trivial extension between $V_{1j}$ and some $V_{2l}$.
\par
First assume that $m_1=m_2=1$, so $V_i$ is an irreducible $X^\circ$-module
for $i=1,2$. Since $X^\circ$ is simple, $u^p$ must act as
an inner automorphism on $X^\circ$. We claim that the order of $u^p$ in its
action on $V$ is bounded above by the order of a unipotent element of
$X^\circ$. Indeed, if the order of $u^p$ is larger than that, then by
Lemma~\ref{lem:innercent} there is a non-trivial central unipotent element in
$X$. On the other hand by Lemma~\ref{lem:ext}(a) the two composition factors
$V_1,V_2$ are not isomorphic as $X^\circ$-modules. This contradicts
Lemma~\ref{lem:unip centr}.
\par
Now $u$ acts by a single Jordan block on both $X^\circ$-composition factors, and
both are tensor indecomposable for $X^\circ$ by Proposition~\ref{prop:new SS}.
Thus each of these two faithful $X$-modules is either trivial or occurs in
Table~\ref{tab:SS97}.

By a result of McNinch \cite[Thm.~1]{McN} the possibilities for non-split
extensions between such $X^\circ$-modules are very
restricted. Namely, he shows that all $X^\circ$-modules of dimension below a
certain explicit bound are either completely reducible or contained on a short
list of exceptions in \cite[Tab.~5.1.1]{McN}. Based on his dimension bounds one
sees that there is no non-trivial extension between a module in
Table~\ref{tab:SS97} and the trivial $X^\circ$-module with $X$ disconnected.
\par
Hence $V_1$ and $V_2$ are non-trivial. Then again combining
Table~\ref{tab:SS97} with \cite[Thm.~1]{McN} we only arrive at the case that
$X$ acts as $A_2.2$ on $V_i$, with $p=2$, $\dim V_i=8$. Here, $V$ has
dimension~16, so $u$ has order~16, while $A_2.2$ has no such element. So
by Lemma~\ref{lem:innercent} there is some non-trivial unipotent element of $X$
centralising $X^\circ$, which implies with Lemma~\ref{lem:unip centr} that the
two composition factors must be isomorphic. So by Lemma~\ref{lem:ext}(a) 
there is no such non-split extension $V$, and hence $V$ is irreducible.
\par
Next consider the case that $m_1=p$. Then we have $p=2$, $X^\circ=A_l$ with
$l\ge2$ and $\dim V_{1j}=l+1$. Then $V_2$ cannot be trivial as argued in the 
preceding case. So if $m_2=1$, then $V_2$ is an irreducible $A_l$- and
$A_l.2$-module on which $u$ acts with a
single Jordan block. By Table~\ref{tab:SS97} this implies that $l=2$ and
$\dim V_2=8$. Recall that here $\dim V_1=6$. A regular unipotent element $u$ of
$\SL(V)$ has order~16. If $u\in X$ then $u^2$ acts as an inner element on
$A_2$. Since the unipotent elements of $A_2$ have order at most~4, there is
an element of order~8 centralising $A_2$ by Lemma~\ref{lem:innercent}. This
contradicts Lemma~\ref{lem:unip centr}.   \par
If $m_2=2$, then $\dim V_{2j}=l+1$ as well. Now unipotent elements of $A_l.2$
have order less than $4(l+1)=\dim V$. So $u$ has bigger order than any unipotent
element of $A_l.2$. Arguing as before this shows with Lemma~\ref{lem:unip centr}
that $V_{11}\cong V_{21}$ (after possibly renumbering), and so also
$V_{11}^*=V_{12}\cong V_{22}$. By untwisting we may assume that $V_{11}$ is
the natural module. But there is no non-trivial extension between the natural
$A_l$-module and itself or its dual.
This final contradiction completes the proof.
\end{proof}

The extension question for the exceptional modules showing up in
Proposition~\ref{prop:irr new}(2) can be discussed in a similar manner:

\begin{prop}   \label{prop:3.4(2)}
 Let $X=A_1^p.\langle u\rangle\le\SL(V)$ with $p\in\{2,3\}$ and $u$ regular
 unipotent in $\SL(V)$. If $V$ has at most two $X^\circ$-composition factors,
 one of which is the $p$-fold tensor product of isomorphic $2$-dimensional
 $A_1$-modules, then $V$ is an irreducible $X$-module.
\end{prop}

\begin{proof}
Assume that $V$ has two $X$-composition factors $V_1<V$ and $V_2=V/V_1$, on
one of which $X^\circ$ acts faithfully. By passing to the
dual, we may assume that this is $V_2$. First consider $p=2$. Then
$\dim V_2=2^p=4$ and $\dim V\ge \dim V_2+1=5$, so the regular unipotent element
$u\in\SL(V)$ has order
at least $8$. But it normalises an $A_1^2$, so by Lemma~\ref{lem:innercent} that
$A_1^2$ is centralised by an element of order~4. Then Lemma~\ref{lem:unip centr}
shows that the two $X^\circ$-composition factors must be isomorphic. As there
are no non-trivial self-extensions for $X^\circ$ by Lemma~\ref{lem:ext}(a),
$V$ is in fact completely reducible as an $X^\circ$-module, and hence an
irreducible $X$-module by Lemma~\ref{lem:soc}, contrary to our assumption.
\par
When $p=3$ then note that there cannot be a non-trivial
extension between the $A_1^3$-module $V_2$ and the trivial module, since the
zero weight is not subdominant to the highest weight of $V_2$. Thus we may
assume that $X$ also acts by a non-trivial semisimple group on $V_1$. Since the
smallest dimension of a faithful $A_1^3$-module is~6,
$\dim V\ge \dim V_2+6\ge14$, and so $u\in\SL(V)$ has order at least~27. We can
now argue exactly as in the case $p=2$, with an element of order~9 centralising
$X^\circ$.
\end{proof}

We can now prove the main result of this section, establishing
Theorem~\ref{thm:main} for the classical algebraic groups $G$ of type $A_l$,
$B_l$ and $C_l$. Recall that it suffices to consider any group isogenous to $G$.

\begin{thm}   \label{thm:SLn}
 Let $G$ be simple of type $A_l,B_l$ or $C_l$ and $X\le G$ be a reductive
 subgroup of the form
 $X=X^\circ\langle u\rangle$ for a regular unipotent element $u$ of $G$ with
 $[X^\circ,X^\circ]\ne1$. Then $X$ does not lie in any proper parabolic subgroup
 of $G$.
\end{thm}

\begin{proof}
As noted in Remark~\ref{rem:char 0} we may assume $p>0$. If $p=2$ then we need
not consider $G=B_l$ as it is isogenous to $C_l$. Now embed $X\le G\le\SL(V)$
via its natural representation. Then by Lemma~\ref{lem:Jordan}, $u$ has a
single Jordan block on $V$. If $X$ lies in some proper parabolic subgroup of
$G$, with unipotent radical $U$, then $X\le N_{\SL(V)}(U)$, which by
Borel--Tits lies in a proper parabolic subgroup of $\SL(V)$. Thus it is
sufficient to establish the result in the case that $G=\SL(V)$. Here the claim
is equivalent to showing that $X$ is irreducible on $V$. For this we may replace
$X^\circ$ by $[X^\circ,X^\circ]$ and therefore assume that $X^\circ$ is a
non-trivial semisimple group. Furthermore, by Remark~\ref{rem:transitive} we
may also assume that $u$ is transitive on the set of simple components of
$X^\circ$ by replacing the latter by one of the $u$-orbits. In particular,
$X^\circ$ acts faithfully on any $X$-composition factor of $V$ on which
it acts non-trivially.
\par
For a contradiction, assume $X$ is reducible on $V$. By passing to a suitable
quotient of~$V$, where $u$ still acts by a single Jordan block, we may assume
that $X$ has exactly two composition factors on $V$. By Lemma~\ref{lem:soc}, any
composition series for $X$ is a socle series for $X^\circ$, which thus also has
two layers.
\par
As $X^\circ\ne1$ there is at least one term in the socle series $0<V_1<V$ on
which $X^\circ$ does not act trivially (which we may assume to be the top layer
$V_2=V/V_1$, by going to the dual if necessary). By
Proposition~\ref{prop:irr new}, $X^\circ$ acts on $V_2$ through $Y_1\cdots Y_r$
with isomorphic semisimple groups $Y_i$ and $r>0$. Let
$V_2=W_1\oplus\cdots\oplus W_r$ be the corresponding decomposition into
$X^\circ\langle u^r\rangle$ modules as in Remark~\ref{rem:prop irr}, with $W_i$
an irreducible $Y_i\langle u^r\rangle$-module and $Y_i$ acting trivially on
$W_j$ for $j\ne i$.
\par
We claim that $X^\circ$ also acts non-trivially on $V_1$. Assume otherwise, so
$\dim V_1=1$. Now $u$ permutes the semisimple factors $Y_1,\ldots,Y_r$
transitively, as well as the $W_1,\ldots,W_r$. For $1\le i\le r$
let $\tilde W_i$ denote the full preimage of $W_i$ in $V$. As
$Y_i\langle u^r\rangle$ stabilises $W_i$, it acts on $\tilde W_i$. If $Y_i$ is
simple, then by Proposition~\ref{prop:ext full}, $u^r$ cannot act with a single
Jordan block on any of the $\tilde W_i$. If $Y_i$ is not simple, so we are in
case~(2) of Proposition~\ref{prop:irr new}, we reach the same conclusion by
Proposition~\ref{prop:3.4(2)}. With $\dim V=r\dim W_1+1$ this contradicts the
fact that by Lemma~\ref{lem:power}, $u^r$ has at least one Jordan block of
size $\dim W_1+1$ on $V$. 
\par
Thus, $X^\circ$ acts faithfully on $V_1$ and on $V_2$. So we have
$V_s=W_{s1}\oplus\cdots\oplus W_{sr}$, where $W_{si}$ is an irreducible
$Y_i\langle u^r\rangle$-module, for $s=1,2$ and $i=1,\ldots,r$. Then
$$\begin{aligned}
  \Ext_{X^\circ}^1(V_1,V_2)
  &=\bigoplus_{i}\Ext_{X^\circ}^1(W_{1i},W_{2i})\oplus
    \bigoplus_{i\ne j}\Ext_{X^\circ}^1(W_{1i},W_{2j})\\
  &=\bigoplus_{i}\Ext_{Y_i}^1(W_{1i},W_{2i})\oplus
    \bigoplus_{i\ne j}\Ext_{X^\circ}^1(W_{1i},W_{2j})
\end{aligned}$$
by Lemmas~\ref{lem:(*)} and~\ref{lem:ext fact}. Now note that the second sum is
zero by comparing highest weights (where we use the fact that if there is a
non-trivial extension between two simple modules for a semisimple algebraic
group then their weights are comparable, see, e.g., \cite[II.2.14]{Ja03}).
Thus, $\Ext_{Y_i}^1(W_{1i},W_{2i})\ne0$ for
some, and hence, all $i$. So $V|_{Y_1}=\tilde W_1\oplus V'$ with $\tilde W_1$ a
non-split extension of $W_{11}$ with $W_{21}$ as a $Y_1$-module and $Y_1$
acting trivially on $V'$. Since $u^r$ normalises $Y_1$, it stabilises this
decomposition of $V$. But $u^r$ acts with $r$ Jordan blocks of size
$\dim V/r=\dim \tilde W_1$ on $V$, so it must act by a single Jordan block on
$\tilde W_1$.    \par
If $Y_i$ is simple, as in case~(1) of Proposition~\ref{prop:irr new}, there is
no such extension $\tilde W_1$ of the two irreducible
$Y_1\langle u^r\rangle$-modules $W_{11},W_{21}$ by
Proposition~\ref{prop:ext full}, giving the desired
contradiction. Hence we must be in the case that $Y_1=A_1^p$, $r=m$, and the
action on the $W_{2i}$ is as in Proposition~\ref{prop:irr new}(2). Again, there
is no such non-split extension by Proposition~\ref{prop:3.4(2)}.
\end{proof}

Thus for Theorem~\ref{thm:main}, as far as simple groups of classical type are
concerned, it remains to consider groups of type~$D_l$.

\section{The case of $\SO(V)$}   \label{sec:SO(V)}

In this section we consider the following situation: $\dim V=2l$ for $l\geq4$,
and $X\leq\SO(V)$ is a (not necessarily connected) reductive subgroup of the
form $X=X^\circ\langle u\rangle$, with $X^\circ$ not a torus, for a regular
unipotent element $u$ of $\SO(V)$. Recall from
Lemma~\ref{lem:Jordan}(c) that regular unipotent elements of $\SO(V)$ have two
Jordan blocks on $V$, of sizes $2l-1,1$ if $p\ne2$, and sizes $2l-2,2$ when
$p=2$.

\subsection{Almost simple groups containing an element with a large Jordan block }

\begin{prop}   \label{prop:l-1}
 Let $p=2$ and $G\le\SL(V)$ be an almost simple algebraic group acting
 irreducibly on $V$ and such that some unipotent element of $G$ has a Jordan
 block of size $\dim V-1$. Then $G=\SL(V)$.
\end{prop}

\begin{proof}
For $G$ a simple group of exceptional type, the irreducible representations of
$G$ whose image contains a unipotent element with a single non-trivial Jordan
block are
classified in \cite[Thm.~1.1]{TZ18}, and only $G_2$ in its 6-dimensional
representation comes up. But here no unipotent elements have blocks of size 5.
(See e.g. \cite{La95}.)
The only almost simple but not simple group of exceptional type is $G=E_6.2$.
Here, the regular unipotent elements of $G$ have order~32, but $G$ has no
faithful irreducible representation of dimension at most~33 (the 27-dimensional
modules for $E_6$ are not invariant under the graph automorphism).  \par
Now assume that $G$ is simple, simply connected, of classical type but not
equal to $\SL(V)$ and let $d$ denote the dimension of its natural module. Here
we may assume $G$ is not an odd dimensional orthogonal group as we will consider
the isogenous symplectic group. The unipotent elements of $G$ have order less
than~ $2d$. On the other hand, writing $n=\dim V$, an element with a Jordan
block of size $n-1$ has order at least $n-1$, so we get $n< 2d+1$.
All irreducible representations of such dimensions are known (see e.g.
\cite[Thms.~4.4 and~5.1]{Lue}), and we find that either $n=d$, or up to
twists one of the following holds:
\begin{itemize}
 \item $G=A_1$ with $d=2$ and $V$ is the tensor product of two 2-dimensional
  modules, or
 \item $G=A_{d-1}$ with $4\le d\le5$ and $V$ is the exterior square of the
  natural module, or
 \item $G=C_{d/2}$ for $d\in\{4,6,8\}$ and $V$ is the spin module, or
 \item $G=D_{d/2}$ for $d=8,10$ and $V$ is a spin module.
\end{itemize}
Looking at the precise order of unipotent elements rules out all cases except
spin  modules for $C_2$, $C_3$  and $D_4$. In the latter two cases, the image
of the representation lies in the 8-dimensional orthogonal group, where there
are no unipotent elements with Jordan blocks of size~7, and for the case of
$G=C_2$, both 4-dimensional representations have image in the symplectic group
where there are also no unipotent elements with a block of size 3. Finally, if
$n=d$, by the above remarks we may assume $V$ is the natural representation of
$G$ ($\ne\SL(V)$), we note that $G$ does not have unipotent elements of the
required Jordan block sizes on $V$ in characteristic~2 (see
\cite[Lemma~6.2]{LS12}).
\par
It remains to consider $G$ almost simple of type $A_{m-1}.2$ or
$D_m.2=\GO_{2m}$. Here, unipotent
elements have order less than $4m$, so as before we conclude $n< 4m+1$. Now
note that the natural module for $A_{m-1}$ is not invariant under the graph
automorphism, nor is its exterior square for $m\ne4$.
While for $m>4$ the spin modules for $D_m$ do not
afford representations of $D_m.2$. Again with \cite{Lue}
we arrive at the Lie algebra for $A_2$, the exterior square of the natural
module for $A_3$, and modules of dimension~$2m$. All of these embed $G$
into a general orthogonal group, but the latter does not have unipotent
elements of the required Jordan type in its natural representation (see
\cite[Lemma~6.2]{LS12}). So none of these lead to examples, completing the
proof.
\end{proof}

\subsection{A reduction result}

\begin{prop}   \label{prop:orthonew1}
 Assume $p>0$. Let $X=X^\circ\langle u\rangle\le\SO(V)$ with $\dim V=2l\ge6$,
 be reductive, $X^\circ\ne1$ semisimple, and $u$ regular unipotent in $\SO(V)$.
 Assume that $X$ does not stabilise any non-zero totally singular subspace of
 $V$. Then one of the following four mutually exclusive cases occurs:
 \begin{enumerate}[\rm(1)]
  \item $X=X^\circ$ is irreducible on $V$; more specifically, either
   $X=\SO(V)$, or $l=4$ and $X=B_3$;
  \item $p=2$, $l$ is even, $X=A_{l-1}.2$ is irreducible on $V$ stabilising a
   pair of complementary totally singular subspaces interchanged by $u$;
  \item there is an orthogonal decomposition $V=V'\perp V''$ into $X$-submodules
   $V',V''$, where $\dim V''=\gcd(p,2)$ and $V'$ is irreducible and tensor
   indecomposable for $X$; or 
  \item $p=2$, $X$ stabilises a 1-dimensional non-singular subspace $V_1$
   of~$V$, it acts as a subgroup of $B_{l-1}$ on $V_1^\perp$ with $u$ having a
   single Jordan block on $V_1^\perp/V_1$, and there exists no
   $X^\circ$-complement to $V_1$ in $V_1^\perp$.
 \end{enumerate}
\end{prop}

\begin{proof}
First assume that $V$ is a decomposable $X$-module. Then by the Jordan block
shape of $u$ we must have $V=V'\oplus V''$ with $\dim V''=\gcd(2,p)$, and $u$
acts with a single Jordan block on both summands. If $X^\circ$ acts
non-trivially on $V'$ then by Theorem~\ref{thm:SLn}, $X$ acts irreducibly on
$V'$. Since $\dim V'>\dim V/2$ then $V'$ must be non-degenerate, so we obtain an
$X$-invariant decomposition $V=V'\perp (V')^\perp$. Lemma~\ref{lem:tens 1}
shows that $V'$ is tensor indecomposable for $X$, so we reach conclusion~(3).
On the other hand, if $X^\circ$ is trivial on $V'$ then it must act faithfully
on $V''$ and hence $X^\circ=A_1$ and $p=2$. In particular, $u$ acts by an inner
automorphism on $X^\circ$ and thus $X^\circ$ contains a regular unipotent
element of $\SO(V)$ by Lemma~\ref{lem:reg inner}, which is impossible as these
have order at least~4.
\par
So now assume that $V$ is an indecomposable $X$-module. In particular, there
is no $X$-invariant non-degenerate non-trivial proper subspace of $V$.
Thus, if $V_1$ denotes a non-zero $X$-invariant subspace of $V$ of
minimal dimension, then either $V_1=V$, that is, $X$ acts irreducibly on $V$, or
$V_1$ is non-singular of dimension~1 and $p=2$. In the latter case $X$ is
contained in the stabiliser of $V_1$, isomorphic to $B_{l-1}$, and
$V_1^\perp$ is the natural module for $B_{l-1}$. Let's first discuss this
situation. Now $u$ is regular unipotent in $B_{l-1}$ by Lemma~\ref{lem:Iulian},
so it has a single Jordan block on $V_1^\perp/V_1$. By Theorem~\ref{thm:SLn}, 
$X$ acts irreducibly on $V_1^\perp/V_1$, and thus $(V_1^\perp/V_1)|_{X^\circ}$
is a direct sum of non-isomorphic irreducible $X^\circ$-modules by
Lemma~\ref{lem:cyc ext}, all
non-trivial as $X^\circ$ acts non-trivially on $V_1^\perp/V_1$. Assume that
$V_1^\perp|_{X^\circ}=V_1\oplus N$. As $N\cong V_1^\perp/V_1$ has no trivial
$X^\circ$-composition factor, this decomposition is $X$-invariant. By
dimension reasons, the irreducible $X$-module $N$ must be non-degenerate, but
this was excluded before.
Thus $V_1$ has no $X^\circ$-complement in $V_1^\perp$, and hence~(4) holds.
\par
Thus we are left to consider the case that $X$ acts irreducibly on $V$. By
Lemma~\ref{lem:cyc ext} then $V|_{X^\circ}=V_1\oplus\cdots\oplus V_m$ is a
direct sum of non-isomorphic irreducible $X^\circ$-modules transitively
permuted by $u$. Then $2l=mr$ with $r=\dim V_1>1$ and $m$ is a power of $p$.
First assume $u$ has Jordan blocks of sizes $2l-1,1$ on $V$. Induction from
Lemma~\ref{lem:power} shows that $u^m$ has $m-1$ Jordan blocks of size~$r$, one
Jordan block of size~$r-1$ and one Jordan block of size~1 on $V$.
Since $u^m$ has the same Jordan blocks on each~$V_i$ by transitivity,
the only compatible solution is $m=1$. If $u$ has Jordan blocks of sizes
$2l-2,2$ on $V$ (and so in particular $p=2$) then Lemma~\ref{lem:power}
shows that either $m=1$, or $u^m$ has $m-2$ blocks of size $r$, and 2 blocks of
sizes $r-1,1$ each. Again this forces $m=2$. In conclusion either $m=1$, or
$m=p=2$. Let us consider these two cases in turn.
\par
If $m=1$, that is, if $V|_{X^\circ}$ is irreducible, then $X^\circ$ is simple
by Lemma~\ref{lem:tens 2} as none of (1)--(3) there can occur here. If
$X^\circ=X=\SO(V)$ we are in case~(1) of our statement. If not, then by
Proposition~\ref{prop:new SS} the only possibility for $l\ge4$ is again the one
given in (1). For $l=3$, using that $X^\circ$ must also be irreducible on the
natural 4-dimensional module $U$ for $A_3\cong D_3$ (e.g. by Borel--Tits),
Proposition~\ref{prop:new SS}
implies that we must have $X^\circ=A_1$ with $p\ge5$ or $X^\circ=C_2=B_2$. But
in neither of these cases does $X^\circ$ act irreducibly on the exterior square
$\Lambda^2(U)=V$, so this case does not occur here.
\par
Now consider the case where $m=p=2$ so that the Jordan blocks of $u^2$ on $V_i$,
$i=1,2$, have sizes~$l-1,1$. We claim that $V_1$ is totally singular. For
otherwise, $V_1$ is non-degenerate and thus so is $V_1^\perp$ which must be
$V_2$, and therefore $X$ is contained in the stabiliser in $\SO(V)$ of the
orthogonal decomposition $V=V_1\perp V_2$, hence in
$\GO(V_1)\GO(V_2).2\cap \SO(V)$. But by \cite[Thm.~B(ii)(a)]{SS97} there is no
reductive maximal subgroup of $\SO(V)$ containing this stabiliser and a regular
unipotent elements of $\SO(V)$.
\par
We thus have that $V_1$ is totally singular as claimed; then so is its image
$V_2$ under $u$. Hence $X$ stabilises a decomposition of $V$ into a direct sum
of maximal totally singular subspaces and thus $X\le\GL_l.2$. According to
Lemma~\ref{lem:Jordan}(d) when $l$ is odd, the regular unipotent elements of
the stabiliser $\GL_l.2$ of such a decomposition have a single Jordan block, so
cannot lie in $\SO(V)$. Thus $X\le\GL_l.2\cap\SO(V)=\GL_l$, which does not
contain elements with a Jordan block of size $2l-2$. On the
other hand, when $l$ is even the stabiliser $\GL_l.2$ contains regular unipotent
elements of $\SO(V)$. Lemma~\ref{lem:tens 2} now implies that $X^\circ$ is
simple, since the case~(3) with $\dim V_1=9$ cannot occur here as $l$ is even.
By Proposition~\ref{prop:l-1} this gives the examples in~(2).
\end{proof}

In what follows we investigate further the case (4) of the preceding result.

\begin{prop}   \label{prop:4.4(4)}
 In the situation of Proposition~$\ref{prop:orthonew1}(4)$, the following hold:
 \begin{enumerate}[\rm(a)]
  \item $X^\circ=Y_1\cdots Y_r$, with pairwise isomorphic factors $Y_i=B_n$,
   $Y_i=C_n$ (with $n\ge1$) or $Y_i=G_2$, permuted transitively by $u$;
  \item there is a decomposition $V_1^\perp/V_1\cong\bigoplus_{i=1}^r U_i$ into
   $Y_i\langle u^r\rangle$-modules $U_i$, irreducible for $Y_i$ and
   transitively permuted by $u$, and on which $u^r$ acts by a single Jordan
   block, with $\dim U_i=2n$ when $Y_i=B_n$ or $C_n$, respectively $\dim U_i=6$
   when $Y_i=G_2$.
 \end{enumerate}
\end{prop}

\begin{proof}
We keep the notation from Proposition~\ref{prop:orthonew1}(4). Recall that here
$p=2$. As $u$ acts by a single Jordan block on $V_1^\perp/V_1$, this is an
irreducible $X$-module by Theorem~\ref{thm:SLn}. So by
Proposition~\ref{prop:irr new} there is a decomposition $X^\circ=Y_1\cdots Y_r$
with $u$ transitively permuting the semisimple factors $Y_i$.

We first show that we are not in case (2) of Proposition~\ref{prop:irr new}.
Suppose the contrary. Then $Y_i=X_{2i-1}X_{2i}$ with $X_j=A_1$, $r=m$, and
$V_1^\perp/V_1 = U_1\oplus\cdots\oplus U_r$, where $U_i$ is an
$X_{2i-1}X_{2i}$-module which is a twist of a tensor product of two natural
modules for $A_1$ and the $U_i$ are transitively permuted by $u$. Let
$\tilde U_i$ denote the full preimage of $U_i$ in $V_1^\perp$, so
$V_1^\perp=\sum\tilde U_i$. The K\"unneth formula \cite[Lemma~3.3.6]{Stew13}
shows that $\Ext_{X_{2i-1}X_{2i}}^1(U_i,k)=0$, whence
$\tilde U_1 = V_1\oplus N_1$, with $N_1\cong U_1$. So there is a
similar decomposition for all of the $\tilde U_i$, leading to a decomposition
of the $X^\circ$-module $V_1^\perp = V_1\oplus (\sum N_i)$, contradicting
Proposition~\ref{prop:orthonew1}(4).

Hence we are in case~(1) of Proposition~\ref{prop:irr new} and all $Y_i$ are
simple. Let us write $\oi:=i\pmod r$. Now
$V_1^\perp/V_1 = U_1\oplus\cdots\oplus U_m$, where the $U_i$ are irreducible
tensor indecomposable $Y_\oi$-modules, transitively permuted by $u$, and $u^m$
acts by a single Jordan block on each. Hence the
possibilities for $(Y_\oi\langle u^m\rangle,U_i)$ are as listed in
Table~\ref{tab:SS97}. Moreover, arguing as in the preceding paragraph we see
that $\Ext^1_{Y_\oi}(U_i,k)\ne 0$. Now by \cite{McN} all pairs
$(Y_\oi\langle u^m\rangle,U_i)$ have $\Ext_{Y_\oi}^1(U_i,k)=0$ except possibly
for $Y_\oi=B_n$, $Y_\oi=C_n$ or $Y_\oi=G_2$ with $\dim U_i$ as claimed, so by 
Proposition~\ref{prop:irr new} we have $r=m$ and we get~(a) and~(b).
\end{proof}

The following proposition treats the special case arising out of
Proposition~\ref{prop:4.4(4)} when $r=1$. By our Lemma~\ref{lem:reg inner}
$X^\circ$ contains a regular unipotent element. This case should have been
treated in \cite{TZ13} but the argument there is incomplete in precisely this
setting. So we have included a proof here. 

\begin{prop}   \label{prop:TZ}
 Let $X$ be a simple algebraic group with $X<G =\SO(V)$, $\dim V =2l\ge8$,
 defined over a field of characteristic~$2$. Assume that $X$ stabilises a
 non-zero totally singular subspace of $V$. Then $X$ does not contain a regular
 unipotent element of $G$.
\end{prop}

\begin{proof}
Assume $u\in X$ is regular unipotent.
Choose a maximal totally singular $X$-invariant subspace $W$. Then the
stabilizer $P = N_G(W)$ has Levi factor $L=\GL(W)\SO(W^\perp/W)$ and the
projection of $X$ into the second factor does not lie in a proper parabolic
subgroup. Also $X\ne A_1$ as the regular unipotent elements in $G$ have order at
least~8.

By Lemma~\ref{lem:reg in Levi}, for $Q = R_u(P)$ the image of $X$ in
$P/Q\cong L$ contains a regular unipotent element of $L$. Thus, if $\dim W>1$,
the projection onto the factor $\GL(W)$ of $L$ is injective when restricted
to $X$, but the order of a regular unipotent element in $\GL(W)$ is strictly
less than the order of $u$. Hence we need only consider the case where
$\dim W=1$. Let $\pi:P\to \SO(W^\perp/W)$ be the natural projection. Since
$\pi(X)$ does not lie in a proper parabolic subgroup of
$\SO(W^\perp/W)$, by \cite[Lemma 2.2]{LT04}, one of:
\begin{enumerate}[(i)]
 \item $(W^\perp/W)\downarrow \pi(X) = W_1\perp W_2\perp\cdots\perp W_t$, with
  all $W_i$ non-degenerate, inequivalent and irreducible $\pi(X)$-modules; or
 \item $\pi(X)$ stabilizes a nonsingular 1-space of $W^\perp/W$.
\end{enumerate}

In the first case, when $t=1$, since $\pi(X)$ contains a regular unipotent
element of $\SO(W^\perp/W)$, Proposition~\ref{prop:new SS} implies that either
$\dim(W^\perp/W) = 8$, $X = B_3$ and the
action of $X$ on $W^\perp/W$ is via a spin module, or $X = \SO(W^\perp/W)$. In
both cases, there are no non-trivial extensions between $W^\perp/W$ and the
trivial $X$-module and we deduce that $X$ (and hence $u$) lies in the Levi
factor, a contradiction.

In the first case, with $t\geq 2$, the Jordan block structure of $u$ then
implies that $t=2$ and we may assume $\dim W_2 = 2$. But then the projection of
$\pi(X)$ to $\SO(W_2)$ is trivial as the latter is a torus, contradicting the
Jordan block structure of $\pi(u)$.

So we now have that $\pi(X)$ lies in the stabilizer of a nonsingular 1-space of
$W^\perp/W$; let $U/W$ be such a subspace, so that $X$ stabilizes the flag
$0<W<U<U^\perp<W^\perp<V$.

Now Proposition~\ref{prop:orthonew1}(4) gives that
$\pi(u)$ has one block on $U^\perp/U$ and so by Table~\ref{tab:SS97} we are left
with the following irreducible actions on $U^\perp/U$:
\begin{itemize}
 \item $\pi(X) = C_{l-2}$ or $B_{l-2}$;
 \item $X = G_2$, $l=5$.
\end{itemize}

In the first case, we deduce that $\pi(X) = B_{l-2}$ (the full stabilizer in
$\SO(W^\perp/W)$ of a non-singular $1$-space), and so $W^\perp/W$ is a
$(2l-2)$-dimensional tilting module for $\pi(X)$.
In particular, there is no extension of this module by a trivial and we find
that the $X$ lies in a Levi factor of $G$. This rules out the case where
$X = B_{l-2}$ or $C_{l-2}$. 

In the second case, we have $G_2Q/Q\leq B_3Q/Q\leq D_4T_1\cong P/Q$, and
the action of $X$ on $W^\perp/W$ is as an $8$-dimensional indecomposable
tilting module (see \cite[Lemma 9.1.1]{LS04}) and as above there is no
non-trivial extension with the trivial module. So $X$ lies in a proper Levi
factor of $G$, and hence cannot contain a regular unipotent element. This final
contradiction completes the proof.
\end{proof}

\begin{prop}   \label{prop:orthonew2}
 Let $X=X^\circ\langle u\rangle\le\SO(V)$ with $\dim V=2l\ge8$ be reductive,
 $X^\circ\ne1$ semisimple, and $u$ regular unipotent in $\SO(V)$. Assume that
 $X$ lies in a proper parabolic subgroup of $\SO(V)$. Then, with $W$ an
 $X$-invariant totally singular subspace of $V$ of maximal possible dimension,
 one of the following two cases occurs:
 \begin{enumerate}[\rm(1)]
  \item $W$ is maximal totally singular, irreducible as an $X$-module, on which
   $u$ acts with a single Jordan block; or
  \item $0<W<W^\perp$, $\dim  W^\perp/W\ge6$, and $X^\circ$ acts
   non-trivially on $W^\perp/W$ and is not completely reducible on $W^\perp$.
 \end{enumerate}
\end{prop}

\begin{proof}
Let $0<W<V$ be as in the statement. So $X$ lies in a proper parabolic subgroup
of $\SO(V)$ with Levi complement $\GL(W)\SO(W^\perp/W)$. By
Lemma~\ref{lem:reg in Levi} the image of $u$ in the Levi factor is again
regular unipotent, in particular $u$ acts by a single Jordan block on $W$. Also
note that $X^\circ$ acts non-trivially on $W^\perp$, as otherwise it would act
trivially on $V/W^\perp\cong W^*$ as well and hence on all of $V$.
If $W=W^\perp$ is maximal totally singular, then as $X$ acts irreducibly on $W$
by Theorem~\ref{thm:SLn}, we arrive at conclusion~(1).
\par
So, we now assume that $0<W<W^\perp$, and thus $X$ is reducible on $W^\perp$.
By Theorem~\ref{thm:SLn} this implies that $u$ cannot have a single Jordan
block on $W^\perp$. Thus, $u$ has exactly two Jordan blocks on $W^\perp$, one
of which has size at most~$\gcd(p,2)$. We set $n:=\dim W^\perp$.

Note that $X$ does not stabilise any non-zero totally singular subspace of
$W^\perp/W$ by the choice of $W$. This then implies that $X^\circ$ acts
non-trivially on $W^\perp/W$ since $u$, lying in a Borel subgroup, has some
totally singular fixed points. In particular $\dim W^\perp/W\ge4$. In fact, if
$\dim(W^\perp/W)=4$ then the image of $X^\circ$ in $\SO_4$ and hence $X^\circ$
itself is either $A_1$ or $A_1^2$. If $X^\circ=A_1$ then $u$ acts on it by an
inner automorphism, so $X^\circ$ contains a regular unipotent element of $G$ by
Lemma~\ref{lem:reg inner}, contradicting the main result of \cite{TZ13}.
Similarly, if $X^\circ=A_1^2\cong\SO_4$ is the full Levi factor and thus
$X=X^\circ$, we conclude by the same argument. Hence we have
$\dim W^\perp/W\ge6$. If $X^\circ$ is not completely reducible on $W^\perp$,
then we arrive at conclusion~(2).

So now assume that $W^\perp$ is a completely reducible $X^\circ$-module.
We will show that this leads to a contradiction. Write
$W^\perp|_{X^\circ}=V_1\oplus\cdots\oplus V_m$ for the decomposition of
$W^\perp$ into its $X^\circ$-homogeneous components. Then $u$ permutes these
components and can have at most
two orbits on $\{V_1,\ldots,V_m\}$, as it has two Jordan blocks on $W^\perp$. 
\par\smallskip\noindent
{\bf Case 1}: We first discuss the case where $u$ is transitive on
$\{V_1,\ldots,V_m\}$. Arguing precisely as in the proof of
Proposition~\ref{prop:orthonew1} the Jordan block shape of $u$ forces either
$m=1$, or $m=p=2$. We consider these two cases in turn.
\par
If $m=1$, that is, if $W^\perp|_{X^\circ}$ is homogeneous, then by
Lemma~\ref{lem:tensor} we find that $\dim W^\perp=4$, as $X^\circ$ is reducible
on $W^\perp$, contradicting $\dim W^\perp/W\ge6$.
\par
If $m=p=2$, one checks that the Jordan blocks of $u^2$ on $V_i$, $i=1,2$, have
sizes~$n/2-1,1$. As $V_i$ is $X^\circ\langle u^2\rangle$-invariant and
homogeneous as an $X^\circ$-module, it must be irreducible for $X^\circ$ by
Lemma~\ref{lem:tensor}. Now $u$ interchanges $V_1$ and $V_2$, so $W^\perp$ is
irreducible for~$X$, a contradiction.
\par\smallskip\noindent
{\bf Case 2}: So now assume that $u$ has two orbits on $\{V_1,\ldots,V_m\}$.
Let $V',V''$ denote the subspaces spanned by these orbits, with
$\dim V''=\gcd(2,p)$. Then $u$ acts with a single Jordan block on each of them.
\par
Consider first the case that $X^\circ$ acts non-trivially on $V'$. Then, since
$u$ has a single Jordan block on $V'$ we obtain by Theorem~\ref{thm:SLn} that
$V'$ is an irreducible $X$-module. Then either $V'= W$ or $W\cap V'=0$. In the
first case $W^\perp/W$ must be isomorphic to $V''$ (as $W\ne W^\perp$),
contradicting that $\dim W^\perp/W\ge6$. If on the other hand $W\cap V'=0$ then
$V'$ is isomorphic to a submodule of $W^\perp/W$. Thus the orthogonal group
$\SO(W^\perp/W)$ of dimension $n-\dim W$ contains a unipotent element with a
Jordan block of size at least $n-\gcd(2,p)$. By the knowledge of possible
Jordan block shapes (see \cite[Lemma~6.2]{LS12}) this is not possible as
$\dim W\ge1$.
\par
Finally, in Case~2 it remains to discuss the situation where $X^\circ$ acts
trivially on~$V'$. Then it must act irreducibly and faithfully on $V''$ and
hence $X^\circ=A_1$ and $p=2$. In particular, $u$ acts by an inner automorphism
on $X^\circ$ and thus $X^\circ$ contains a regular unipotent element of
$\SO(V)$ by Lemma~\ref{lem:reg inner}, which is impossible by order
considerations.
\end{proof}

\subsection{Proof of Theorem~\ref{thm:main} for $\SO(V)$, $\dim V=2l$}

\begin{thm}   \label{thm:Dl}
 Let $X=X^\circ\langle u\rangle\le G:=\SO(V)$, where $\dim V=2l\ge8$, be a
 reductive subgroup, with $u$ a regular unipotent element of $G$. Assume that
 $[X^\circ,X^\circ]\ne1$. Then $X$ does not lie in any proper parabolic
 subgroup of $G$.
\end{thm}

\begin{proof}
It suffices to prove the claim for $[X^\circ,X^\circ]\langle u\rangle$, and
hence we may and will assume that $X^\circ=[X^\circ,X^\circ]$ is semisimple.
Moreover, by Remark~\ref{rem:transitive} we may assume that $X^\circ$ is the
product over a single $u$-orbit of simple components.
Assume that $X$ lies in a proper parabolic subgroup $P$ of $G=\SO(V)$.
Then there is an $X$-invariant flag $0<W\le W^\perp<V$ with $W$ totally singular
and dual to $V/W^\perp$ as an $X$-module, and $W^\perp/W$ non-degenerate. We
choose $P$ such that $\dim W$ is maximal possible (and so $\dim W^\perp/W$ is
minimal). Hence, we are in the setting of Proposition~\ref{prop:orthonew2}.
\par
Consider first the case~(1). That is, $X$ is contained in the stabiliser of a
maximal totally singular subspace $W=W^\perp$ and acts as an irreducible
subgroup of $\GL(W)$ on $W$ and (hence) on its dual $V/W$. Note that
$X^\circ$ acts faithfully on $W$.   \par
Let us first consider the case when $p=2$. Assume that $u$ acts with different
orders on $W$ and on $V$. Then a non-trivial power $v$ of $u$, which we may
take to be an involution, is in the kernel of the representation on $W$. Letting
$\vhi_1:P\to\GL(W)$ denote the projection into the Levi factor of $P$, we thus
have $[v,X^\circ]$ lies in the normal subgroup $X^\circ$, and has image
$\vhi_1([v,X^\circ])=[\vhi_1(v),\vhi_1(X^\circ)]=1$, so in fact $[v,X^\circ]=1$,
that is, $v$ centralises $X^\circ$. As $u^2$ has two Jordan blocks $J_1$,
the element $v$ has Jordan blocks $J_2^a\oplus J_1^b$ with $a\ge1$, $b\ge2$ on
$V$. (Here $J_i$ denotes a Jordan block of size $i$.) Then by
\cite[Thm.~3.1]{LS12},setting $C = C_{\SL(V)}(v)$, we have that
$C/R_u(C)\cong \GL_a\GL_b$ and the action of this group on $V$ is via two
copies of the natural module for $\GL_a$ and one copy of the natural module for
$\GL_b$. As $X\subseteq C$, the dimensions of the $X$-composition factors
on $V$ (two factors of dimension~$l$) must be obtained as a refinement of the
dimensions $a,a,b$ coming from
the action of $\GL_a\GL_b$, which is not possible. So $u$ has the same order on
$W$ as on $V$. As $u$ has a Jordan block of size $l$ on $W$, and blocks of sizes
$2l-2,2$ on $V$, we must have $l=2^f+1\ge5$. So we have proved that $l$ is odd
when $p=2$.   \par
Now let $p$ be arbitrary again and write
$W|_{X^\circ}=W_1\oplus\cdots\oplus W_r$ for the decomposition of $W$ into
pairwise non-isomorphic irreducible $X^\circ\langle u^r\rangle$-modules as in
Remark~\ref{rem:prop irr}, with $u$ transitively permuting the summands. Note
that since $l=\dim W$ has to be odd when $p=2$ we cannot have $m=2r$ in
the notation of Proposition~\ref{prop:irr new}. As $W,V/W$ are dual to each
other, then $(V/W)|_{X^\circ}\cong W_1^*\oplus\cdots\oplus W_r^*$ with $W_i^*$
dual to $W_i$. Let $V_1$ be the $X^\circ\langle u^r\rangle$-submodule of $V$
with composition factors $W_1$ and $W_1^*$ (if $V$ is not completely reducible
as an $X^\circ\langle u^r\rangle$-module, this exists by Lemma~\ref{lem:(*)}).
Then the transitive action of $u$ yields
$V|_{X^\circ}=\bigoplus_{i=1}^r u^{i-1}(V_1)$. From the block structure of $u$,
as in the proof of Proposition~\ref{prop:orthonew1} this implies that $r=1$
or $r=p=2$. But the latter cannot occur as $l$ is odd when $p=2$. So $r=1$,
and $V_1=V$ is an extension of an irreducible $X^\circ$-module by its dual, on
both of which $u^r=u$ has a single Jordan block. Assume that $V$ is completely
reducible. If $W\cong W^*$ is self-dual, so $V$ is homogenous, then it is
an irreducible $X$-module by Lemma~\ref{lem:tensor}, a contradiction. Else, the
decomposition $V=W\oplus W^*$ is $X$-invariant, contrary to the block structure
of $u$.  \par
So $V$ is a non-trivial extension, and since by Lemma~\ref{lem:ext}(a) there
are no self-extensions for simple groups, $W$ cannot be a self-dual
$X^\circ$-module.
Moreover, still by Proposition~\ref{prop:irr new}, $X^\circ$ is now either
simple or $X^\circ=A_1^p$ with $p\le3$. The second possibility is ruled out as
$W$ is not self-dual. The possible non-self-dual $W$ for $X^\circ$ simple are
listed in Table~\ref{tab:SS97}, and we find that $X^\circ=A_l$ with $l\ge2$.
But according to \cite[Cor.~1.1.1]{McN}, there is no non-trivial extension of
a twist of the natural module of $A_l$ with its dual.
\par
Thus we are in case~(2) of Proposition~\ref{prop:orthonew2}. By the choice of
$W$ and since $\dim W^\perp/W\ge6$, we have that $X$ acts as in (1)--(4) of
Proposition~\ref{prop:orthonew1} on $W^\perp/W$. Also, by
Lemma~\ref{lem:reg in Levi}, $u$ acts with a single Jordan block on $W$ as well
as on $V/W^\perp$, and with two Jordan blocks on $W^\perp/W$. 

\par\smallskip\noindent
{\bf Case 1:} First assume that $X^\circ$ acts non-trivially and hence
faithfully on $W$ and thus that $\dim W\ge2$. By
Proposition~\ref{prop:orthonew2}(2), $X^\circ$ acts non-trivially and hence
faithfully on $W^\perp/W$ as well, which has dimension at least~6.
\par\noindent
{\bf Case 1a:} We first discuss the case where $p=2$. Let $\vhi_i$, $i=1,2$, be
the projections of $P$ into the two factors of the Levi subgroup $\GL(W)$,
$\SO(W^\perp/W)$ respectively. Since $\vhi_1(X)$ is irreducible on $W$ by
Theorem~\ref{thm:SLn}, it
cannot lie in a proper parabolic subgroup of $\GL(W)$, and by the choice of
$W$, neither is $\vhi_2(X)$ contained in a proper parabolic subgroup of
$\SO(W^\perp/W)$. Write $m=\dim W$ and so $\dim W^\perp/W=2(l-m)$.
We know that $\vhi_1(u)$ is a single Jordan block, and $\vhi_2(u)$ has a block
of size $2(l-m)-2$ and one of size 2, by Lemma~\ref{lem:reg in Levi}. Since
$m<l$ by assumption, $\vhi_1(u)$ has order smaller than $u$, so some power
$u^s\ne1$ with $s>1$ lies in $\ker\vhi_1$; we choose $s$ minimal with this
property. As before, we see that $u^s$ must centralise $X^\circ$.
But then $\vhi_2(u^s)=1$ as well, as otherwise $\vhi_2(u^s)$ is a non-trivial
unipotent element of $\SO(W^\perp/W)$ centralised by $\vhi_2(X)$, whence by
Borel--Tits, $\vhi_2(X)$ lies in a proper parabolic subgroup of
$\SO(W^\perp/W)$, which is
not the case. Note that no smaller power of $u$ lies in $\ker\vhi_2$ as
otherwise that element would (as before) centralise $X^\circ$, forcing
$\vhi_1(X)$ to lie in a proper parabolic of $\GL(W)$, again a contradiction.
So $\vhi_1(u),\vhi_2(u)$ have the same order~$s$.
\par
Recall that $u$ has a single block of size $m$ on $W$ and blocks of sizes
$2l-2m-2$ and $2$ on $W^\perp/W$. Also, $u$ has two blocks on $W^\perp$, one of
size~1 or~2. But the first possibility is ruled out as
$(u-1)^{2l-2m-2} W^\perp\subseteq W$ and so $(u-1)^m(u-1)^{2l-2m-2}W^\perp=0$.

We now show that $u$ has order $2s$ on $W^\perp$, that is, order twice as large
as its order on $W$ (and on $W^\perp/W$). Let $f$ be minimal such that
$2^f\geq m$, so $f$ is also minimal so that $2^f\geq 2l-2m-2$.
Hence we have $2m>2^f\geq m$ and $4l-4m-4>2^f\geq 2l-2m-2$. In particular,
$2^{f+1}\geq 2l-m-2$. Since also $2^f=2^{f-1}+2^{f-1}<m+2l-2m-2=2l-m-2$, the
order of $u$ in its action on $W^\perp$ is $2s$. In particular, $u^s$ acts as
an involution
$J_2^a\oplus J_1^b$ on $W^\perp$, with $a\geq 1$ and $b\geq 2$ (as $s>1$).
As before, the centraliser of $u^s$ in $\SL(W^\perp)$ has composition factors of
dimensions $a,a,b$ on $V$, and as $u^s$ centralises $X$, the $X$-composition
factors on $W^\perp$ must be obtained as a refinement of this.

Note that since $a\geq 1$ and $b\geq 2$, there are at least three composition
factors and so $X$ cannot act irreducibly on $W^\perp/W$, ruling out
configurations~(1) and~(2) of Proposition~\ref{prop:orthonew1}. Considering the
cases in~(3) and~(4), we see that $X$ has composition factor dimensions
on $W^\perp$ among
$$\{m, 2l-2m-2,2\}\mbox{ and } \{m,2l-2m-2,1,1\}.$$
As $l\ge8$ and $2l-2m\ge6$, the first case yields $a = m=2l-2m-2$ and $b=2$.
In the second case, the natural module for $\GL_a$ must remain irreducible for
$X$, as else there are five composition factors, and so either $a=1$, or
$a=m=2l-2m-2$ and $b=2$.

If $a=1$, then $u^s$ has one block of size~2 and the remaining blocks of size~1
on $W^\perp$. By Lemma~\ref{lem:power} this can only happen if $2l-m-2=2^c+1$
for some $c$, so $u$ has order $2^{c+1}$ on $W^\perp$ and by the previous
analysis, its order on $W$ and on $W^\perp/W$ is $2^c$ (so $c=f$ as above).
But this implies $2^{c-1}<m<2^{c-1}+1$, a contradiction.

Hence we have $b=2$ and $a=m=2l-2m-2$ and $u^s$ has exactly two blocks of size~1
on $W^\perp$ (coming from the block $J_2$ of $u$ on $W^\perp$), while the
one block of size $2l-m-2$ produces only blocks of size~$2$ for $u^s$. So
$2l-m-2$ is a power of $2$, say $2l-m-2=2^c$. Thus, $u$ has order $2^c$ on
$W^\perp$ and the order of $u$ on $W$ is $2^{c-1}$, so $2^{c-1}\geq m$ and
$2^{c-1}\geq 2l-2m-2$, forcing $m=2^{c-1} = 2l-2m-2$.

We claim that neither of the possible actions of $X$ on $W^\perp/W$ (as in
Proposition~\ref{prop:orthonew1}(3) and~(4)) is consistent with this.
First suppose we have the configuration of Proposition~\ref{prop:orthonew1}(3),
where $W^\perp/W = V_1\oplus V_2$, $X$ is irreducible on $V_1$ and $\dim V_2=2$.
Let $\tV_i$ be the $X$-submodules of $W^\perp$ such that $\tV_i/W = V_i$,
$i=1,2$. On each of these $X$ acts reducibly and so $u$ has at least two Jordan
blocks. Hence we have $\dim \tV_i^u\ge2$ and $\tV_1\cap \tV_2=W$, whence
$\dim(\tV_1^u\cap \tV_2^u)=\dim W^u=1$ as $u$ has a single Jordan block on $W$,
which gives $\dim(\tV_1+\tV_2)^u\ge3$, contradiction.

So finally we are left to consider the case where $X$ acts on $W^\perp/W$ as in
Proposition~\ref{prop:orthonew1}(4). Here we have
$0\le W\le V_1\le V_1^\perp\le W^\perp$, with $\dim W = m$,
$\dim V_1/W = 1 = \dim W^\perp/V_1^\perp$ and $\dim V_1^\perp/V_1 = m$. Recall
that $u$ has order $2m=2^c$ on $W^\perp$ and thus the same order on the
codimension~1 subspace $V_1^\perp$, which is again twice the order of $u$ on
$W$ and on $W^\perp/W$. So again $u^s$ acts as an
involution on $V_1^\perp$ and the $X$-composition factors are of dimensions
$m,1,m$. This is only consistent with the $\GL_a\GL_b$ analysis if $b=1$. This
is the final contradiction settling Case~1a.
\par\noindent
{\bf Case 1b:} So now we have $p>2$. As $\dim(W^\perp/W)\ge6$, we may apply
Proposition~\ref{prop:orthonew1} to the action of $X$ on $W^\perp/W$, and as
$p>2$ we are in either case~(1) or~(3). If we are in case~(1), $X^\circ$ acts
as $B_3$ on $W^\perp$, so $u$ acts by an inner automorphism on a component of
$X^\circ$ and we are done by Lemma~\ref{lem:reg inner} and \cite{TZ13}.
\par
In case~(3) of Proposition~\ref{prop:orthonew1}, we have
$W^\perp/W=V_1\oplus V_2$, with preimages $\tV_1,\tV_2$ in $W^\perp$.
Now $X^\circ$ acts non-trivially on both $\tV_i$, normalised by $u$, so
$u$ has two Jordan blocks on each $\tV_i$, by Theorem~\ref{thm:SLn}.
Counting fixed points on $\tV_1+\tV_2$ as in Case~1a, we reach a
contradiction.

\par\smallskip\noindent
{\bf Case 2:} Now assume that $X^\circ$ acts trivially on $W$. Let $W_0$ be the
$u$-invariant subspace of $W$ of codimension~1. Note that $W_0^\perp/W_0$ is
non-degenerate and $u$ acts as a regular unipotent element of
$\SO(W_0^\perp/W_0)$ by Lemma~\ref{lem:reg in Levi} and the image of $X$ lies
in a proper parabolic subgroup of this orthogonal group. So it suffices to
derive a contradiction in that situation, whence henceforth we assume
$\dim W=1$. Again, as $\dim(W^\perp/W)\ge6$, we may apply the conclusion of
Proposition~\ref{prop:orthonew1} to the image of $X$ in
$\SO(W^\perp/W)=\SO_{2l-2}$ (using our assumption that $\dim W=1$).
\par\noindent
{\bf Case 2a:} In case~(1) of Proposition~\ref{prop:orthonew1} again we have
$X=X^\circ$ as in Case~1b, a situation that was handled in \cite{TZ13}.
\par\noindent
{\bf Case 2b:} In case~(2) of Proposition~\ref{prop:orthonew1} we have
$X^\circ=A_{l-2}$, with $l\ge4$ as $\dim(W^\perp/W)\ge6$. Note that $V$ is then
a completely reducible $X^\circ$-module
since there are no extensions between the natural and the trivial module for
$A_{l-2}$, so $W^\perp/W$ is isomorphic to an $X^\circ$-submodule $M$ of
$W^\perp$. Assume that $M\cap M^\perp\ne0$. Then, by dimension
reasons, this intersection must be one of the two non-isomorphic irreducible
$X^\circ$-summands. But then $M/(M\cap M^\perp)$ has a non-degenerate
$X^\circ$-invariant form and thus is a self-dual $A_{l-2}$-module, which it is
not. Thus, $M$ is a non-degenerate $X^\circ$-submodule of $V$, and $M^\perp$ is
its 2-dimensional orthogonal complement. Since $M,M^\perp$ are both sums of
homogeneous $X^\circ$-components of $V$, the decomposition $V=M\perp M^\perp$ is
$u$-invariant, making $M$ an irreducible $X$-module. Thus $W\le M^\perp$ is a
1-dimensional totally singular subspace, but the 1-dimensional fixed points of
unipotent elements of $\GO(M^\perp)$ are non-singular.
\par\noindent
{\bf Case 2c:} In case~(3) of Proposition~\ref{prop:orthonew1}, we have an
$X$-stable decomposition $W^\perp/W=V_1\oplus V_2$ with $X$ irreducible on
$V_1$. Write $\tV_i$ for the full preimage of $V_i$ in $W^\perp$, $i=1,2$, both
$X$-submodules of $V$. By Theorem~\ref{thm:SLn}, as $\tV_1$ is reducible
for $X$, we deduce that $u$ has two blocks on $\tV_1$. If $u$ has more than one
Jordan block on $\tV_2$ as well, then counting fixed points as in Case~1a we
obtain a contradiction. So $u$ has a single Jordan block on $\tV_2$,
whence $\tV_2^u=W$ and $(W^\perp)^u=\tV_1^u$.
\par
Now first assume that $p\ne2$ so that $\dim V_2=1$. Using that
$W^\perp=\tV_1+\tV_2$ we have
$W^\perp/(W^\perp)^u=W^\perp/\tV_1^u\cong \tV_1/\tV_1^u+\tV_2/W$,
and a dimension count then shows the sum is direct. Hence
$\dim (W^\perp/(W^\perp)^u)^u\ge2$, contradicting the Jordan block structure of
$u$ on~$V$. Similarly, when $p=2$ and hence $\dim V_2=2$, consider
$M:=\ker((u-1)^2|_{W^\perp})$. Then $\dim M\cap\tV_1\ge3$,
$\dim M\cap\tV_2=2$, and these two intersect in $W$.
By assumption, $u$ has one fixed point on $V/\ker((u-1)^2)$, but on
$W^\perp/M=(\tV_1+M)/M\oplus (\tV_2+M)/M$ it has a two-dimensional
fixed point space, giving a contradiction.
\par\noindent
{\bf Case 2d:} In case~(4) of Proposition~\ref{prop:orthonew1}, we have $p=2$
and $W^\perp/W$ has $X$-invariant subspaces $0< \bar V_1<\bar V_1^\perp$ with
$\bar V_1$ non-singular of dimension~1, such that $u$ acts with
one Jordan block on $\bar V_1^\perp/\bar V_1$. By Proposition~\ref{prop:4.4(4)}
we may decompose $X^\circ=X_1\cdots X_r$ into a product of simple groups all
isomorphic to either $B_m$, $C_m$ or $G_2$, with $\dim V=2l=2mr+4$, where we set
$m:=3$ in the case of $G_2$. Furthermore, $r$ is a power of~2, with $r\ge2$
since otherwise we are done by Lemma~\ref{lem:reg inner} and
Proposition~\ref{prop:TZ}. As $u$ has Jordan blocks of sizes $2mr+2,2$ on $V$,
$u^2$ has two Jordan blocks of size $mr+1$ and two of size~1.

Note that $v:=u^2$ acts trivially on the 2-dimensional full preimage $V_1$ of
$\bar V_1$ in $V$. By our
assumption, $V_1$ cannot be totally singular. Let $N\le V_1$ be a 1-dimensional
non-singular subspace. Then $X^\circ\langle v\rangle$ lies in the stabiliser
$N_{\SO(V)}(N)$ isomorphic to $B_{l-1}\cong C_{l-1}$, so in $\Sp(N^\perp/N)$.
We claim that $v$ has Jordan block sizes $mr+1,mr+1$ on $N^\perp/N$. Indeed,
as it has block sizes $mr+1,mr+1,1,1$ on $V$, the only other possibility would
be $mr,mr,1,1$ (note that all odd block sizes most occur an even number of
times). But by \cite[Thm.~6.6]{LS12} the centraliser of such an element has
reductive part of its centraliser containing an $\Sp_2$, while the reductive
part of the centraliser of $v$ in $\SO(V)$ is just a torus, a contradiction.

Now the $X^\circ\langle v^{r/2}\rangle$-composition factors of $N^\perp/N$ are
the $U_i$ and two trivial modules. Thus, by self-duality, $N^\perp/N$ has a
submodule $M$ of codimension~1, and this is the sum of submodules of dimension
at most $2m+1$ (namely either the $U_i$ or extensions of some $U_i$ by a trivial
module). But $v^{r/2}=u^r$ has block sizes $2m+1,2m+1,2m\,(r-2$ times$)$ on
$N/N^\perp$, so at least one block of size $2m+1$ on $M$. By
Theorem~\ref{thm:SLn}, this contradicts the fact that $M$ has no irreducible
$X^\circ\langle v^{r/2}\rangle$-submodules of that dimension.
\end{proof}

\section{Exceptional types}   \label{sec:exc}

In this section we consider algebraic groups defined over $k$ of
characteristic~$p>0$. See Remark~\ref{rem:char 0}.

We will make extensive use of the known data on unipotent elements in simple
algebraic groups of exceptional type, including element orders and power maps
given in \cite{La95} and structure of centralisers described in
\cite{LS12}. We follow the notation in \cite{La95} for the labelling of
unipotent classes. In particular, if the class of $u$ is denoted by some Dynkin
diagram, then $u$ is a regular element in a Levi subgroup of that Dynkin type.

In the course of our proof we will require precise knowledge on the existence
and conjugacy classes of complements to $R_u(C_G(x))$ in $C_G(x)^\circ$,
for certain unipotent elements $x$, as in the next result.

\begin{lem}   \label{lem:cent split}
 Let $x\in G$ be unipotent and let $Y\le C_G(x)$ be connected reductive, where
 $$(G,p,Y,[x])\in\big\{(E_7,3,A_1^3,4A_1),(E_8,2,A_2,E_6(a_1)),
   (E_8,2,A_2,E_6)\big\}.$$
   Then there exists a connected reductive group $C\leq C_G(x)$ such that
   $C_G(x)^\circ = R_u(C_G(x)).C$ and
   $Y$ lies in a conjugate of $C$.
\end{lem}

\begin{proof}
Throughout we write $R:=R_u(C_G(x))$. The existence of a complement to $R$ in
$C_G(x)^\circ$ follows from \cite[17.6]{LS12}. We now turn to the proof of the
remaining assertions.   \par
Consider first $G = E_7$, with $p=3$, $Y=A_1^3$, $x$ a unipotent element of type
$4A_1$
and $Y\leq C_G(x)^\circ= RC$, where $C$ is a long root $C_3$-subgroup of $G$
(see \cite[Tab.~22.1.2]{LS12}). By \cite[Thm.~5]{LS96} there exist two classes
of such $C_3$-subgroups in $G$, coming from the two non-conjugate $A_5$ Levi
factors of $G$. By \cite[Cor., p.2]{LS96}, there exists a unique class of
complements to $R$ in $RC$. For both classes of $C_3$-subgroups, each
non-trivial $C$-composition factor of $R$ occurs as a composition factor of
$\wedge^j(W)$, for some $1\leq j\leq 3$, where $W$ is the natural
$6$-dimensional $C$-module. (See \cite[Tab.~8.2]{LS96}.) There exists a unique
class of $A_1^3$-subgroups of $C_3$, and restricting each of the given
irreducible $C_3$-modules to such an $A_1^3$-subgroup we find that the
composition factors of $Y$ on $R$ have highest weights among $\omega_i$,
$i=1,2,3$ (the fundamental dominant weights of $A_1^3$),
$\omega_i+\omega_j$ for $1\leq i<j\leq 3$, $\omega_1+\omega_2+\omega_3$ and
the zero weight. Using \cite[II.2.14]{Ja03} we have
$H^1(Y,\bar Q)=0$ for all $Y$-composition factors $\bar Q$ of $R$ and hence by
\cite[Prop.~3.2.6]{Stew13}, there exists a unique class of complements to $R$
in $RY$. Thus there exists $g\in C_G(x)$ with $Y \leq C^g$ as claimed.  \par
In the other two cases, we have $G=E_8$, with $p=2$, $Y=A_2$, and $x$ is a
unipotent
element of type either $E_6(a_1)$ or $E_6$, and $C_G(x)^\circ = RC$, where
$C = \bar {A}_2$, respectively $\bar {G}_2$, long root subgroups. The action of
a long root $A_2$ on $R$ has composition factors the natural, dual or trivial
module, and so in case $C=\bar A_2$, we have a unique
class of complements to $R$ in $RY=RC$, allowing to conclude.

In the case $C=\bar G_2$ we must argue slightly differently because here there
is a $6$-dimensional $C$-composition factor $V$ of $R$ with  $H^1(G_2,V)\ne 0$.
Now $C_G(x)$ lies in a parabolic subgroup $P=QL$ of $G$ with $R\leq Q$.
Moreover, considering the labelled diagram of the class of~$x$ (see
\cite[Tab.~22.1.1]{LS12}) and applying \cite[Thm.~17.4]{LS12}, we see that we
may take $P$ to be a $D_4$-parabolic subgroup of $G$. There
exists a composition series of $Q$ as an $[L,L]$-module, all of whose terms are
$8$-dimensional $D_4$-modules and trivials. The subgroup $C$ is uniquely
determined up to conjugacy in $[L,L]$ and, by \cite[Lemma 9.1.1]{LS04},
each such irreducible upon restriction to $C$ is
the indecomposable tilting module $U$ with $\rad U/\Soc U$ the 6-dimensional
irreducible $C$-module. By \cite[Prop.~\S E.1]{Ja03},
$H^1(C,U)=0$. Hence, all complements to $C$ in $QC$ are conjugate and
by considering the action of $Y=A_2$ on $Q$, we have the same statement for
$Y$. Arguing as in the previous cases now yields the claim.
\end{proof}

\begin{thm}   \label{thm:exc}
 Let $G$ be a simple algebraic group of exceptional type and
 $X=X^\circ\langle u\rangle\le G$ a reductive subgroup with $u$ a regular
 unipotent element of $G$ and $[X^\circ,X^\circ]\ne1$. Then $X$ does not lie in
 a proper parabolic subgroup of $G$.
\end{thm}

\begin{proof}
Let $X$ be as in the assertion and assume that $X$ lies in a proper parabolic
subgroup of $G$. Then we have $u\notin X^\circ$ by \cite[Thm.~1.2]{TZ13}. Also,
by passing to $[X^\circ,X^\circ]\langle u\rangle$ we may assume that $X^\circ$
is semisimple. We will need to consider the image of $X$ in Levi factors, and
for this
throughout we write $\bar X$ for the quotient of $X$ by its largest normal
unipotent subgroup (which, being finite, is centralised by $X^\circ$).
Note that $X^\circ$ maps isomorphically to a subgroup of $\bar X$ on
which the image of $\langle u\rangle$ then acts faithfully.

By Remark~\ref{rem:transitive}, we may moreover assume that $u$ has a single
orbit on the set of simple components of $X^\circ$, and, by
Lemma~\ref{lem:reg inner} it does not act as an inner automorphism on~$X^\circ$.
In particular, $\rnk(X^\circ)\ge p$.
On the other hand, as $X$ lies in a proper parabolic subgroup of $G$,
$\rnk(X^\circ)< \rnk(G)$. This already rules out the case $G=G_2$. Furthermore,
regular unipotent elements of $F_4$ are also regular in $E_6$ under the natural
embedding, and proper parabolic subgroups of $F_4$ lie in such of $E_6$, hence
it suffices to prove our result for $G=E_6,E_7$ or $E_8$. The orders of regular
unipotent elements in these groups for small primes are given in
Table~\ref{tab:order reg}.

\begin{table}[htb]
\caption{Orders of regular unipotent elements}   \label{tab:order reg}
$$\begin{array}{c|cccc}
 & p=2& 3& 5& 7\cr
\hline
 G_2&  8&  9\cr
 F_4& 16& 27\cr
 E_6& 16& 27& 25\cr
 E_7& 32& 27& 25& 49\cr
 E_8& 32& 81& 125& 49\cr
\end{array}$$
\end{table}

\smallskip\noindent
{\bf Case~1:} We first consider the case that $v:=u^p$ acts by an inner
automorphism on $X^\circ$ and $X^\circ$ does not contain elements of order
$|v|$, so in particular $|v|>p$.   \par
Then by Lemma~\ref{lem:innercent}, $X^\circ$ is centralised by a unipotent
element $v'$ of this order. We discuss this situation by comparing the list of
centralisers of unipotent elements \cite[\S22]{LS12} and the list of unipotent
element orders \cite[Tab.~5--9]{La95}.
\par
Let first $G=E_6$ and consider the case $p=2$, where we have $|v'|=|v|=8$ and
$v'$ centralises $X^\circ$. By \cite[Tab.~22.1.3]{LS12} and \cite[Tab.~5]{La95}
unipotent elements of order~8 centralising a group of semisimple rank at
least~2 lie in class $D_4$, with reductive part of the centraliser of type
$A_2$. Since $\rnk(X^\circ)\ge2$ we conclude $X^\circ=A_2$. Now $u^2$ acts as
an inner automorphism of $X$ and so $X$ centralises $u^8$, which lies in the
class $2A_1$, by \cite[Tab.~D]{La95}. Moreover, using \cite[Tab.~22.1.3]{LS12},
we see that the full connected centraliser of $u^8$ has a reductive complement
$C$ to $R = R_u(C_G(u^8))$, a $B_3$-subgroup of $G$ generated by long root
elements of $G$. Now $X^\circ\leq RC$ and we consider the possible embedding of
$X^\circ$ in $C\cong C_G(u^8)^\circ/R$. By \cite[Lemma~2.2]{LT04}, $X$ must lie
in a proper parabolic subgroup of $C$ and by rank considerations we find that
$X$ lies in an $A_2$-parabolic subgroup of $C$. As $C$ is generated by long
root subgroups of $G$, the Levi factor $A_2$ is also generated by long root
subgroups of $G$. Now arguing as in Lemma~\ref{lem:cent split}, we find that
$X^\circ$ is a long root $A_2$-subgroup of $G$, that is, a Levi factor of $G$.
(The $A_2$ Levi factor of $C$ acts on $R$ with composition factors the natural,
dual or trivial module for $A_2$.)
But now the centraliser of $X^\circ$ is an $A_2A_2$ and so $X$ normalizes an
$A_2^3$ subgroup of $G$. But there is no such example in \cite[Thm.~A]{SS97}.
For $p=3$ we have $|v'|=9$ and $\rnk(X^\circ)\ge3$,
and again by \cite[Tab.~22.1.3]{LS12} and \cite[Tab.~5]{La95} there is no
possibility. For $p=5$ we have $|v'|=5$, contrary to our assumption.
\par
For $G=E_7$ and $p=2$ we have $|v'|=16$ but all centralisers of such elements
have semisimple rank at most~1 by \cite[\S22]{LS12}. When $p=3$ and so $|v'|=9$,
the unipotent classes $A_3$, $(A_3+A_1)^{(1)}$, $(A_3+A_1)^{(2)}$, $D_4$ and
$D_4(a_1)$ need to be discussed. Here the semisimple parts of the centralisers
have type $B_3A_1$, $B_3$, $A_1^3$, $C_3$, $A_1^3$, respectively. As $X^\circ$
is contained in one of those, and $u$ has a single orbit on its set of simple
components, $X^\circ$ must be of type $A_1^3$. Now $X=A_1^3\langle u\rangle$
contains $u^3$ acting as an inner element on $A_1^3$, and so $u^9$ centralises
$A_1^3$. But by \cite[Tab.~D]{La95}, $u^9$ lies in class $4A_1$, with
semisimple part of its centraliser a $C_3$-subgroup generated by long root
subgroups, by \cite[\S22]{LS12}. By Lemma~\ref{lem:cent split}, $X^\circ$
must be a long root $A_1^3$. There are two classes
of Levi subgroups $A_1^3$ in $E_7$. Using Borel--de Siebenthal one sees that
one is centralised by an $A_1^4$, the other by a $D_4$. Thus in any case,
$X^\circ C_G(X^\circ)$ contains a subgroup of $E_7$ of maximal semisimple rank,
so the normaliser of $X^\circ$ and all of its overgroups are reductive.
But by \cite[Thm.~A]{SS97} there are no such subgroups containing a regular
unipotent element. For $p=5$ we have $|v'|=5$, contrary to our assumption.
\par
Finally, assume $G=E_8$. For $p=2$ we have $|v'|=16$. Only the 17 unipotent
classes
$$E_6(a_1),\,D_6,\,E_6,\,E_6(a_1)+A_1,\,E_7(a_3),\,E_8(b_6),
  \,D_7(a_1),\,D_7(a_1)^{(2)},\,E_6+A_1,$$
$$E_7(a_2),\,E_8(a_6),\,D_7,\,E_8(b_5),\,E_7(a_1),\,E_8(a_5),\,E_8(b_4),
  \,E_8(a_4)$$
of $G$ contain elements of order~16. Of these, only the first three cases have a
semisimple part of the centraliser of rank at least~2, of type $A_2$, $B_2$ and
$G_2$ respectively. So $X^\circ=A_2$ or $A_1^2$, with $u$ acting by the graph
automorphism. In the second case, $u^4$ centralises $X^\circ$ and lies in class
$D_7(a_2)$, but the latter has centraliser of rank~1. So in fact $X^\circ=A_2$.
By \cite[Tab.~22.1.1]{LS12}, the remaining possible subgroups $A_2$, $G_2$ of
$C_G(v')$ are
generated by long root subgroups, and by Lemma~\ref{lem:cent split}, $X^\circ$
is contained in one of them. Now all subgroups of type $A_2$ of these are again
generated by long root subgroups, hence so is $X^\circ$. Thus, the centraliser
of $X^\circ$ is of type $E_6$, and so $u$ acts on an $E_6A_2$, whose normaliser
is a maximal subgroup of $E_8$. But since this does not appear in
\cite[Thm.~A]{SS97}, its normaliser does not contain regular unipotent elements.
\par
If $p=3$ then $|v'|=27$, but no unipotent element of order bigger than~9 has a
centraliser of semisimple rank at least~3. If $p=5$ then $|v'|=25$, but none of
the seven unipotent classes having centraliser of semisimple rank at least~5
contains elements of order~25. Finally, for $p=7$ we have $|v'|=7$, which
is not allowed here.
\par
\smallskip\noindent
{\bf Case~2:} We next consider the case that $v:=u^p$ acts by an inner
automorphism on $X^\circ$, and that $X^\circ$ does contain an element of order
$|v|$.  \par
When $G=E_6$ and $p=2$, then $X^\circ$ is a semisimple subgroup with an element
of order~8 having a non-trivial graph automorphism transitively permuting the
simple factors. Therefore, $X^\circ$ is one of $A_4,A_5,D_4,D_5,G_2^2$. By
assumption $X$ lies in a proper parabolic subgroup of $G$ with Levi factor $L$,
and thus $L$ contains one of the above groups, with the image of $u$ inducing a
non-trivial graph automorphism of order~$2$.
By rank considerations, only $X^\circ=A_4$, $D_4$ and $G_2^2$ might occur. The
smallest faithful representation of $\bar X=A_4.2$ has dimension~10, so this
cannot occur inside~$A_5$. This representation embeds $A_4.2$ into $\GO_{10}$,
but not into $\SO_{10}$ by the block structures given in Lemma~\ref{lem:Jordan}.
Also, the smallest faithful representation of $G_2^2$ has dimension~$12$, too
large for any proper Levi subgroup. So in fact we must have $X^\circ=D_4$
inside a $D_5$-parabolic. Using \cite[3.2.6]{Stew13} one can check that this
embedding is into a Levi factor, and so $X^\circ=D_4$ is a Levi
subgroup of $G$. By \cite[Tab.~22.1.3]{LS12}, no non-trivial unipotent element
of $G$ has a $D_4$ in its centraliser, so $C_G(X^\circ)$ is a torus, and then in
fact it must be the centre $T_2$ of a Levi subgroup of type $D_4$. As $u$ acts
on $C_G(X^\circ)$, Proposition~\ref{prop:Aut T} implies that $u^4$ must
centralise $T_2$. But $u^4$ lies in the class $2A_2+A_1$ and has centraliser of
rank $1$ by \cite[Tab.~22.1.3]{LS12}, a contradiction.
\par
For $G=E_6$ and $p=3$ with $|v|=9$, we have that $X^\circ$ contains elements of
order~9 and
has an outer automorphism of order~3, so $\bar X=D_4.3$. But the smallest
faithful representation of $X$ has dimension~24, which is too large for
containment in any proper parabolic subgroup of $G$. For $p=5$ where
$|v|=5$ the only option is that $X^\circ=A_1^5$. But no proper parabolic
subgroup has a Levi factor containing a group $X$ with $\bar X=A_1^5.5$.  \par
For $G=E_7$ with $p=2$ the semisimple group $X^\circ$ has an element of order~16
and a non-trivial graph automorphism, whence $X^\circ\in\{D_6,E_6\}$. But
clearly no Levi factor of a proper parabolic subgroup of $E_7$ can contain
$X$ with $\bar X=D_6.2$ or $E_6.2$.
For $p=3$ with $|v|=9$ the only possibilities with a graph automorphism of
order~3 are $\bar X=D_4.3$, $B_2^3.3$ and $G_2^3.3$. All could only lie in a
proper parabolic subgroup of type $E_6$. But $E_6$ has no maximal rank subgroups
$B_2^3$ or $G_2^3$ by Borel--de Siebenthal. When $\bar X=D_4.3$,
\cite[Thm.~5]{LS96} shows that
$X^\circ$ is a Levi factor of $G$. Again by Borel--de Siebenthal there is a
subgroup $A_1^3$ centralising $X^\circ=D_4$, so $N_G(D_4)\ge D_4A_1^3$ is
reductive. But by \cite[Thm.~A]{SS97}, there is no positive-dimensional maximal
reductive subgroup of $G$ containing a regular unipotent element.
For $p=5$ again the only possibility is $X^\circ=A_1^5$. The only proper
parabolic subgroups whose Levi factor might contain $X$ with $\bar X=A_1^5.5$
are those
of type $D_6$. The list in \cite[Thm.~B]{SS97} shows that there is no maximal
reductive subgroup of $D_6$ containing a regular unipotent element of $D_6$ and
such an $A_1^5$.
\par
For $G=E_8$ and $p=2$ we have $|v|=16$. The only semisimple groups of rank at
most~7 with a unipotent element of order~16 and an even order graph automorphism
are $D_6$, $D_7$ and $E_6$. Now for $X^\circ=D_6$ or $D_7$ the element $v$ of
order~16 acts as an inner element $x$ of order~8. Thus $X^\circ$ is
centralised by the element $vx^{-1}$ of order~16, which is not possible. Assume
$X^\circ=E_6$ and $u$ induces a graph automorphism on $E_6$. There is only
one class of subgroups $E_6$ in $E_8$ by \cite[Thm.~5]{LS96} and hence $X^\circ$
is a Levi factor of $G$. Again, $u$ normalises the centraliser of
such an $E_6$, hence a subgroup $E_6A_2$, and as above this is not possible by
\cite[Thm.~A]{SS97}. If $p=3$ then $|v|=27$, and there is no possible case.
When $p=5$ or $p=7$, then $X^\circ$ must have at least~$p$ simple components, whence
$X^\circ=A_1^p$. Now for $p=5$, the group $A_1^5$ does not contain elements of order~25,
so we have $p=7$ and $X^\circ=A_1^7$. The only proper parabolic subgroup of $E_8$
with a Levi factor containing $X$ with $\bar X=A_1^7.7$ is of type $E_7$. By
\cite[Thm.~4 and Tab.~17 and~18]{LT18}, such an $A_1^7$ lies in a
Levi factor $E_7$. Now the centraliser of that $A_1^7$ contains the $A_1$
centralising the $E_7$-Levi subgroup. So again the normaliser of $X^\circ$ in
$E_8$ has maximal semisimple rank, and \cite[Thm.~A]{SS97} shows that this
cannot contain regular unipotent elements.
\par
\smallskip\noindent
{\bf Case~3:} Finally, consider the case that $u^p$ is not inner. Then either
$X^\circ$ has at least $p^2$ components, or there are $p$ components and on each
of them $u^p$ induces a graph automorphism of order~$p$. Either possibility
forces $\rnk(X^\circ)\ge p^2$, so $p=2$. Furthermore, $v:=u^4$ must act by an
inner automorphism on $X^\circ$.
When $G=E_6$ then the possibilities are $X^\circ=A_1^4$ or $A_2^2$. In the
first case, by Lemma~\ref{lem:innercent} there exists an element of order~4
centralising an $A_1^4$, which is not possible by \cite[\S22]{LS12}. In the case
$X^\circ=A_2^2$ and $u$ acts as an outer automorphism of order~4. The only Levi
factor possibly containing a subgroup $X$ with $\bar X=A_2^2.4$ is of type
$D_5$, but $A_2^2.4$ contains elements of order~16 while $D_5$ does not have
such elements. When $G=E_7$ then $|v|=8$; none of the groups $A_1^4$,
$A_2^2$ and $A_3^2$ contains elements of that order, so by
Lemma~\ref{lem:innercent} there is an element of $G$ of order~8 centralising
such a subgroup, which is not the case by \cite{La95} and \cite{LS12}. Finally,
when $G=E_8$ then again $|v|=8$, and the candidates for $X^\circ$ are $A_1^4$,
$A_2^2$ and $A_3^2$. Assume $\bar X=A_1^4.4$, then $u^4$ acts by an inner
automorphism
so $u^8$ centralises $X^\circ$. But $u^8$ lies in class $D_4(a_1)+A_2$ (there is
a misprint in \cite[Tab.~D]{La95}), and its centraliser does not contain an
$A_1^4$. Similarly, if $\bar X=A_2^2.4$ then $u^{16}$, in class $4A_1$,
centralises $X^\circ$, which is not possible. The same argument rules out
$\bar X=A_3^2.4$. This completes our case distinction and thus the proof.
\end{proof}

Theorem~\ref{thm:main} now follows by combining
Theorems~\ref{thm:SLn}, \ref{thm:Dl} and~\ref{thm:exc}.

\section{Regular unipotent elements in almost simple groups}   \label{sec:almost simple}
We now extend our main result to the case of regular unipotent elements in
cosets of simple groups in almost simple groups.

\begin{exmp}   \label{exmp:reg exc}
 The regular unipotent elements in a coset $xG^\circ\ne G^\circ$ of an almost
 simple group $G$ of ``exceptional type'' can be realized as follows:  \par
 (a) The group $D_4.3$ occurs as a subgroup of $F_4$ in a natural way (see e.g.
 \cite[Ex.~13.9]{MT}). Now according to \cite[Tab.~4]{La95} the only unipotent
 class of $F_4$ for $p=3$ containing elements of order~27 is the class of
 regular unipotent elements. Also, the regular unipotent elements in an outer
 coset of the disconnected group $D_4.3$ have order~27 (see
 \cite[Tab.~8]{Ma93}). Thus, they are regular unipotent elements of $F_4$.\par
 (b) Similarly, the disconnected group $E_6.2$ occurs inside the normaliser of a
 Levi subgroup of type $E_6$ inside $E_7$. Again by \cite[Tab.~7]{La95} the only
 unipotent class of $E_7$ for $p=2$ containing elements of order~32 is the class
 of regular unipotent elements, and since regular unipotent elements in the
 outer coset of $E_6.2$ have order~32 (see \cite[Tab.~10]{Ma93b}), they must be
 regular unipotent elements of $E_7$. 
\end{exmp}

We then obtain the following consequence of Theorem~\ref{thm:main}:

\begin{cor}   \label{cor:dis}
 Let $G$ be almost simple of type $A_l.2,D_l.2$ or $E_6.2$ with $p=2$, or
 of type $D_4.3$ with $p=3$, and $X=X^\circ\langle u\rangle\le G$ be a
 reductive subgroup with $u$ a regular unipotent element of $uG^\circ$ and
 $[X^\circ,X^\circ]\ne1$. Then $X$ does not lie in any proper subgroup $P$ of
 $G$ such that $P^\circ$ is a parabolic subgroup of $G^\circ$.
\end{cor}

\begin{proof}
In each case, we embed $G$ in a simple algebraic group $H$; namely, $A_{2l}.2$
and $D_l.2$ embed into $H:=\SL(V)$ via their natural representation,
$G=A_{2l-1}.2$ embeds into $H:=D_{2l}$ (see remarks before
Lemma~\ref{lem:Jordan}), and $D_4.3$, $E_6.2$, embed into $H=F_4$, $E_7$
respectively under the embeddings given in Example~\ref{exmp:reg exc}.
Applying Lemma~\ref{lem:Jordan} and Example~\ref{exmp:reg exc}, we have that
the embedding sends regular unipotent elements in an outer coset of $G^\circ$
to regular unipotent elements of $H$.  Now, if $X$ lies in a proper subgroup
$P$ of $G$ with $P^\circ$ a parabolic subgroup of $G^\circ$ with
$Q=R_u(P^\circ)$, then $X\le N_H(Q)$, and by the Borel--Tits theorem, the
latter lies in a proper parabolic subgroup of $H$. Thus, in all cases our claim
for the almost simple group $G$ follows from Theorem~\ref{thm:main} for the
simple group $H$.
\end{proof}

Note that the type of subgroups $P$ allowed for in the preceding statement are
those given by the most general possible definition of ``parabolic subgroups of 
an almost simple group''.

\section{Regular unipotent elements in normalisers of tori}   \label{sec:tori}
Here, we show that if one removes the hypothesis that $[X^\circ,X^\circ]\ne1$ in
Theorem~\ref{thm:main}, the conclusion is no longer valid. More generally,
for a simple group $G$ we investigate the structure of torus normalisers in
$G$ that contain a regular unipotent element and lie in some proper
parabolic subgroup of $G$.

\subsection{Torus normalisers in $\SL(V)$}   \label{subsec:tori SL}

\begin{prop}   \label{prop:torus SL}
 Let $X=T\langle u\rangle\le\SL(V)$ where $T$ is a torus and $u$ is unipotent
 with a single Jordan block. Then all weight spaces of $T$ on $V$ have the same
 dimension $d$. Moreover $X$ is contained in a proper parabolic subgroup of
 $\SL(V)$ if and only if $d>1$.
\end{prop}

\begin{proof}
Since $u$ normalises $T$, the weight spaces of $T$ on $V$ are permuted by $u$.
Moreover this action must be transitive as otherwise $u$ would have at least
two Jordan blocks on $V$. Thus, they all have the same dimension, and if they
are 1-dimensional, $V$ is an irreducible $X$-module and so $X$ does not lie in
any proper parabolic subgroup of $G$.  \par
Now assume the common dimension of the weight spaces is $d>1$ and set
$m=\dim(V)/d$. Since $u$ has $p$-power order, the number of weight spaces, $m$,
is a
$p$-power. It follows by Lemma~\ref{lem:power} that $u^m$ acts with a
single Jordan block (of size $d$) on each weight space. In particular, $T$
centralises $u^m\ne1$ and thus $X\le C_G(u^m)$ lies in a proper parabolic
subgroup of $G$ by the Borel--Tits theorem (\cite[Rem.~17.16]{MT}).
\end{proof}

Groups as in the previous result do in fact exist:

\begin{prop}   \label{prop:torus SL II}
 Let $n=p^ad$ be an integer, where $a>0$. There exists a $p^a-1$-dimensional
 torus $T\le\SL(V)$, where $\dim V=n$, with $d$-dimensional weight spaces
 on~$V$, normalised by a unipotent element $u\in\SL(V)$ with a single Jordan
 block. Moreover, $T\langle u\rangle$ lies in a proper parabolic subgroup of
 $\SL(V)$ if and only if $d>1$.
\end{prop}

\begin{proof}
Decompose $V=V_1\oplus\cdots\oplus V_m$ into a direct sum of $m:=p^a$
subspaces of dimension~$d$. Let $x\in\SL(V)$ be the permutation matrix for
a permutation sending an ordered basis of $V_i$ to an ordered basis of $V_{i+1}$
for $i=1,\ldots,m$, where $V_{m+1}:=V_1$. Then $x$ has order~$m$. For
$i=1,\ldots,m$ let $T_i\le\GL(V_i)$ be the torus of scalar matrices and
$u_1\in\SL(V_1)$ a unipotent
element with a single Jordan block. Set $X':=\langle T_1,u_1,x\rangle\le\GL(V)$.
As $x$ permutes the $T_i$ transitively and $x^m=1$ we have that $X'$ is the
wreath product of $T_1\times\langle u_1\rangle$ with $\langle x\rangle$, and we
can write elements of $X'$ as $(x_1,\ldots,x_{m};x^j)$ for
$x_i\in (T_1\langle u_1\rangle)^{x^{i-1}}$ and some $j$. Then the element
$u:=(u_1,1,\ldots,1;x)$ has $m$th power $u^m=(u_1,\ldots,u_1;1)$ which has $m$
Jordan blocks of size $d$ on $V$. But then $u$ must have a single Jordan
block on $V$ by Lemma~\ref{lem:power}. Now with $T:=T_1\cdots T_m\cap\SL(V)$ the
subgroup $T\langle u\rangle$ is as in Proposition~\ref{prop:torus SL} and thus
the claim follows.
\end{proof}

\subsection{Torus normalisers in $\Sp(V)$ and $\GO(V)$}

We next discuss those classical groups in which regular unipotent elements
have a single Jordan block on the natural module, that is, the types $B_l,C_l$
and $D_l.2$ (see Lemma~\ref{lem:Jordan}).

For this, note that weight spaces for non-zero weights of a torus $T$ in
$\Sp(V)$ respectively $\SO(V)$, are totally isotropic, respectively totally
singular, and weight spaces for weights $\chi,\eta$ with $\chi\ne-\eta$, are
orthogonal to each other. To see this for $\SO(V)$, let $Q$ be the quadratic
form and $\beta$ the associated bilinear form on~$V$. Let $\chi\ne0$ be a
weight of $T$ with weight space $V_\chi$. Then for $v\in V_\chi$, we have
$Q(v) = Q(tv) = Q(\chi(t)v) = \chi(t)^2Q(v)$, for all $t\in T$. Since there
exists $t\in T$ with $\chi(t^2)\ne1$ we find $Q(v)=0$. So non-zero weight
spaces are indeed totally singular.
Further, for $\chi$ and $\eta$ two weights of $T$ and $v\in V_\chi$,
$w\in V_\eta$, we have $\beta(v,w)=\beta( tv,tw)=\chi(t)\eta(t)\beta(v,w)$ for
all $t\in T$, and if $\beta(v,w)\ne 0$ then $\chi = -\eta$. The argument for
$\Sp(V)$ is completely analogous.

\begin{lem}   \label{lem:torus p=2}
 Let $G=B_l$ or $C_l$ with $l\ge1$ and $X=T\langle u\rangle\le G$ where
 $T\ne1$ is a torus and $u$ is regular unipotent in $G$. Then $p=2$.
\end{lem}

\begin{proof}
Assume $p\ne2$.
Write $V$ for the natural module of $G$ and let $V=\bigoplus_{i=1}^r V_i$ be its
$T$-weight space decomposition. Note that we have $r\ge2$ since $Z(G)^\circ=1$.
Now $u$ permutes the $V_i$ and hence their corresponding weights. As $u$ acts
as a single Jordan block on $V$ by Lemma~\ref{lem:Jordan}, this action must be
transitive, so $r$ is a power of $p$. Since $r\ge2$ there is at least one
non-zero weight $\chi$. As we are in $\Sp(V)$ or $\SO(V)$, then $-\chi$ is also
a weight, so both $\chi$ and $-\chi$ lie in one $u$-orbit, contradicting
that $p\ne2$.
\end{proof}

In the case $p=2$, by the exceptional isogeny between $B_l$ and $C_l$ we need
not consider type $B_l$. 

\begin{prop}   \label{prop:torus CD}
 Let $p=2$, $G=\Sp(V)$ or $\GO(V)$ and $X=T\langle u\rangle\le G$ with $T\ne1$
 a torus and $u$ having a single Jordan block on $V$. If
 $V=V_1\oplus\cdots\oplus V_r$ is the $T$-weight space decomposition then the
 $V_i$ are totally isotropic, respectively totally singular, and permuted
 transitively by~$u$. Moreover, up to renumbering, there is an orthogonal
 decomposition $V=(V_1\oplus V_{r/2+1})\perp\ldots\perp (V_{r/2}\oplus V_r)$.
\end{prop}

\begin{proof}
As $u$ has a single Jordan block, it permutes the $V_i$ transitively and so
$r$ is a 2-power. All $V_i$ are totally isotropic, respectively totally
singular, by the remarks before Lemma~\ref{lem:torus p=2}, and orthogonal to
all other weight spaces that do not have opposed weight. Further, if $\chi$ is
a weight of $T$ then so is $-\chi$, and thus for a suitable numbering, $V_i$
and $V_{i+r/2}$ have opposed weights, for $i=1,\ldots, r/2$. Thus we obtain the
claimed orthogonal decomposition.
\end{proof}

\begin{exmp}   \label{exmp:torus CD}
 The situation nailed down in Proposition~\ref{prop:torus CD} does give rise to
 examples within proper parabolic subgroups. To see this, let $p=2$, $G=\GO(V)$
 with $\dim V=2l$, where $l=2^fm$ with $m>1$ odd. By Lemma~\ref{lem:Jordan},
 for any odd $m$ the stabiliser $\GL_m.2$ in $\GO_{2m}$ of a maximal totally
 singular subspace contains a unipotent element $v$ with a single Jordan block.
 This normalises $T_1:=Z(\GL_m)$, and $v^2\ne1$ centralises $T_1$. Now embed
 $$T_1\langle v\rangle\times\cdots\times T_1\langle v\rangle\le
   H:=\GO_{2m}\times\ldots\times\GO_{2m}\le\GO_{2l}=\GO(V)$$
 ($2^f$ factors). The normaliser of $H$ in $\GO_{2l}$ contains an element $x$
 cyclically permuting the factors. Set $u:=(v,1,\ldots,1)x$. By construction
 $u^{2^f}=(v,\ldots,v)$ has $2^f$ Jordan blocks of size $2m$, so by
 Lemma~\ref{lem:power}, $u$ has a single Jordan block on $V$, it normalises
 $T:=T_1^{2^f}$, and $u^{2^{f+1}}\ne1$ centralises $T$. Thus,
 $T\langle u\rangle$ lies in a proper parabolic subgroup of $\GO(V)$, as in
 Proposition~\ref{prop:torus CD}. Since $D_l.2\le C_l$, this also provides
 examples in $C_l$.
\end{exmp}

\subsection{Torus normalisers in $\SO(V)$}
Here we consider tori in $D_l$ normalised by a regular unipotent element.

\begin{prop}   \label{prop:torus D}
 Let $X=T\langle u\rangle\le\SO(V)$ with $\dim V=2l\ge8$, $T\ne1$ a torus and
 $u$ regular unipotent in $\SO(V)$. Then $p=2$ and if
 $V=V_1\oplus\cdots\oplus V_r$ is the $T$-weight space decomposition then up
 to renumbering the $V_i$, we have one of:
 \begin{enumerate}[\rm(1)]
  \item $r=2$, $l$ is even, $u$ interchanges $V_1,V_2$ and $u^2$ acts with
   Jordan blocks of sizes $l-1,1$ on both $V_1$ and $V_2$;
  \item $u$ permutes $V_1,\ldots,V_{r-1}$ transitively (so $r=2^s+1$ for some
   $s\ge1$) and $V_r$ is the $0$-weight space, with $\dim V_r=2$; or
  \item $\langle u\rangle$ acts transitively on $\{V_1,\ldots,V_{r-2}\}$ and on
   $\{V_{r-1},V_r\}$, so $r=2^s+2$ for some $s\ge0$, and $V_{r-1}$ and $V_r$ are
   $1$-dimensional weight spaces for opposed weights.
 \end{enumerate}
\end{prop}

\begin{proof}
Let $V=\bigoplus_{i=1}^r V_i$ be the decomposition of $V$ into non-zero
$T$-weight spaces. Note that we have $r\ge2$ since $Z(\SO(V))^\circ=1$. From the
block structure of $u$ it follows that $\langle u\rangle$ has at most two orbits
on the set of $V_i$. In addition, the sum of the weight spaces in one of the
orbits is of dimension at most~$2$. Since
we are in $\SO(V)$, if $\chi$ is a weight of $T$ on $V$, then so is $-\chi$.
Now first assume that $p$ is odd. Then $\chi$ and $-\chi$ can only lie in the
same $u$-orbit if $\chi=0$. So $u$ has two orbits on the set of weight spaces,
one of length~$r-1$ and the other of length~1. There is a non-zero weight
$\chi$ in one of the orbits; the weight space of $-\chi$ then lies in the other
orbit. This forces $\dim V=2$, contrary to our assumption. 
\par
Thus we have $p=2$. First assume $u$ permutes the $V_i$ transitively. Then
$u^r$ stabilises each $V_i$, and has same block sizes $n_1,\ldots,n_s$ on each
of them. Since $r$ is a 2-power, the blocks of $u$ on $V$ then have sizes
$rn_1,\ldots,rn_s$, whence $r\le2$ and so $r=2$. Since $V_1,V_2$ are both
totally singular, then $X$ is contained in the stabiliser of a decomposition of
$V$ into a sum of two maximal totally singular subspaces. If $l$ is odd, then
this stabiliser in $\SO(V)$ fixes each $V_i$ (see \cite[Lemma~2.5.8]{KL90}).
Thus $l$ is even, $u$ interchanges $V_1$ and $V_2$ and $u^2$ has Jordan blocks
as claimed in~(1). \par
Next assume that $u$ permutes $V_1,\ldots,V_{r-1}$ transitively. Then without
loss of generality $\dim V_r=2$. If $V_r$ is not the 0-weight space,
then the opposite weight space must be one of the other $V_i$, so $r=2$, and
$\dim V=4$, contradicting our assumption. So we arrive at~(2).
\par
Finally, assume that $u$ permutes $V_1,\ldots,V_{r-2}$ transitively. Then
$\dim V_{r-1}=\dim V_r=1$ and the corresponding weights are opposed and
interchanged by $u$, which is~(3)
\end{proof}

\begin{exmp}   \label{exmp:torus orth}
We show that the cases in Proposition~\ref{prop:torus D} do give rise to
examples within
proper parabolic subgroups. So let $p=2$.\par
(1) Let $l$ be even, $T$ be the 1-dimensional central torus of $\GL_l$ inside
the stabiliser $\GL_l.2$ in $\SO_{2l}$ of a pair $V_1,V_2$ of maximal totally
singular subspaces. Thus $T$ acts by scalars on both $V_1,V_2$.
Then $T$ is normalised by the outer elements of $\GL_l.2$ interchanging
$V_1,V_2$. Now by Lemma~\ref{lem:Jordan}, a regular unipotent element $u$ in
the outer coset of $\GL_l.2$ has Jordan blocks of sizes $2l-2,2$, hence is
regular unipotent in $\SO_{2l}$. Then $X:=T\langle u\rangle$ lies in the
centraliser of the non-trivial unipotent element $u^2\in\GL_l$ (non-trivial
as soon as $l-1\ge2$), thus inside a proper parabolic subgroup. This is an
example of~(1) in Proposition~\ref{prop:torus D}. \par
(2) Let $H=\GO_{2l-2}\GO_2\cap\SO(V)$ be the stabiliser of an orthogonal
decomposition of $V$, where $\dim V=2l$ with $l=2^s+1$. Then by
\cite[Thm.~B(ii)(a)]{SS97}
there is a subgroup $T\langle u\rangle\le H$, with $T$ a maximal torus and
$u$ a regular unipotent element of $\SO(V)$. We number the weights
$\chi_1,\ldots,\chi_{2l}$ of $T$ on $V$ such that $u$ acts as the permutation
$(1,2,\ldots,2l-2)(2l-1,2l)$ on these. For
$T_1=\ker\chi_{2l-1}\cap\ker\chi_{2l}$, the group $T_1\langle u\rangle$ is an
example for case~(2). On the other hand by taking the direct product of $\GO_2$
with a subgroup of $\GO_{2l-2}$ as constructed in Example~\ref{exmp:torus CD},
and intersecting with $\SO(V)$ we find an example for~(3), and as in
part~(1) we see that both lie inside proper parabolic subgroups. 
 \par
(3) The example for $\SO_6=\SL_4$ in Proposition~\ref{prop:torus SL II} falls
into case~(3); this can be seen from the weight spaces on the two modules, as
the natural module for $\SO_6$ is the wedge square of the natural module for
$A_3$.
\end{exmp}

We are not aware of examples of torus normalisers in disconnected groups
$A_l.2$ containing outer regular unipotent elements and lying in a proper
parabolic subgroup.

\subsection{Torus normalisers in simple exceptional groups}
Finally, we investigate the case of exceptional groups.

\begin{prop}   \label{prop:Aut T}
 Let $T$ be a torus and $u\in\Aut(T)$ of prime power order $p^a$. Then
 $\dim T\ge p^{a-1}(p-1)$.
\end{prop}

\begin{proof}
We have $\Aut(T)\cong\GL_n(\ZZ)$ with $n=\dim T$. If $u\in\GL_n(\ZZ)$ has
order $p^a$ then it must have an eigenvalue $\zeta$ which is a primitive
$p^a$th root of unity. But then all Galois conjugates of $\zeta$ are also
eigenvalues of $u$, and there are $\varphi(p^a)=p^{a-1}(p-1)$ of these.
\end{proof}

\begin{rem}   \label{rem:weyl group}
 Let $T$ be a torus in a connected reductive group $G$ and $u\in G$ a unipotent
 element acting non-trivially on $T$. Then $p$ divides the order of the Weyl
 group of $G$. Indeed, by assumption $u\in N_G(T)/C_G(T)$ is non-trivial. As
 $L=C_G(T)$ is a Levi subgroup of $G$ and $N_G(T)\le N_G(L)$, the claim
 follows with \cite[Cor.~12.11]{MT}.
\end{rem}

\begin{prop}   \label{prop:torus exc}
 Let $G$ be simple of exceptional type and $X=T\langle u\rangle<G$ with
 $T$ a non-trivial torus and $u$ a regular unipotent element of $G$. Then
 one of the following holds:
 \begin{enumerate}[\rm(1)]
  \item $G=E_6$, $p=3$, $\dim T=2$; or
  \item $G=E_7$, $p=2$, $\dim T=1$.
 \end{enumerate}
\end{prop}

\begin{proof}
The regular unipotent element $u$ induces a non-trivial automorphism $\bar u$
of $T$, so by the previous remark, $p$ divides the order of the Weyl group of
$G$.

Combining the $p$-power map on unipotent classes \cite[Tab.~D and~E]{La95} and
the structure of centralisers \cite[\S22]{LS12} we have compiled in
Table~\ref{tab:tori cent} a list of the dimensions of maximal tori in the
centralisers $C_G(u^{p^i})$ for $i\ge1$ and $u^{p^i}\ne1$.

\begin{table}[htb]
\caption{Ranks of centralisers $C_G(u^{p^i})$, $i\ge1$}   \label{tab:tori cent}
$$\begin{array}{c|cccc}
 & p=2& 3& 5& 7\cr
\hline
 G_2& 0,1& 0\\
 F_4& 0,0,2& 0,3\\
 E_6& 0,1,{4}& {2},5& 2\\
 E_7& {1},{2},{4},6& 0,3& 1& 2\\
 E_8& 0,1,2,4& 0,3,7& 0,7& 1\\
\end{array}$$
\end{table}

Now first consider $G=G_2$. Then $\dim T\le2$, so $\bar u$ has order at most~4
when $p=2$, respectively~3 when $p=3$, by Proposition~\ref{prop:Aut T}.
Hence $u^4$, respectively $u^3$, must centralise~$T$, which by
Table~\ref{tab:tori cent} implies $\dim T=1$ and $p=2$. But in that case,
$\bar u$ has order at most~2, so $u^2$ centralises $T$ and we reach a
contradiction to Table~\ref{tab:tori cent}.   \par
When $G=F_4$, then $\dim T\le4$ and by Proposition~\ref{prop:Aut T}, $\bar u$ has
order at most~8. Again by Table~\ref{tab:tori cent} this gives that $\bar u$ has
order~8 and $\dim T\le3$, contradicting the bound in
Proposition~\ref{prop:Aut T}.
\par
For $G=E_6$ with $\dim T\le6$ we have $|\bar u|\le 8,9,5$ when $p=2,3,5$
respectively. For $p=2$, using Proposition~\ref{prop:Aut T} and
Table~\ref{tab:tori cent} we find that $\dim T=4$ and $\bar u$ of order~8 is the
only possibility. Here $u^8$ lies in class $2A_1$ by \cite[Tab.~D]{La95}, and
its centraliser has rank~4. Let $L$ be an $A_1^2$-Levi subgroup of $G$
containing $u^8$; it has connected centre $Z(L)^\circ$ of dimension~4, so this
must be the torus $T$ in $C_G(u^8)$. Now $u$ normalises $T$, so it also
normalises $L=C_G(T)$, and thus $[L,L]=A_1^2$. If $u$ acts by an inner
automorphism on $A_1^2$, then by Lemma~\ref{lem:innercent} there is an element
of order~16 centralising $A_1^2$, but the only elements of $G$ of that order are
regular, a contradiction. Therefore, it acts by a graph automorphism on the
$A_1^2$ and $u^2$ is inner and hence some element of order $8$ centralises
$A_1^2$. Again by \cite{LS12} and \cite{La95} there is no
element of order~8 in $G$ with such a centraliser. So this does not occur.
Next, for the case $p=3$ using Proposition~\ref{prop:Aut T} and
Table~\ref{tab:tori cent} as above, only $\dim T=2$ with $\bar u$ of order~3
remains. So $v:=u^3$, in class $D_4(a_1)$ by \cite[Tab.~D]{La95}, centralises
$T$, and we reach case~(1) of the statement. The case $p=5$ is not possible by
Table~\ref{tab:tori cent}.
\par
For $G=E_7$ with Proposition~\ref{prop:Aut T} and Table~\ref{tab:tori cent} and
arguing as above we are left with the case that $p=2$ and either $\dim T=4$ and
$\bar u$ has order~8, or $\dim T=2$ and $|\bar u|=4$, or $\dim T=1$ and
$|\bar u|=2$. The last case occurs in the conclusion, so we need to exclude the
former two. If $\dim T=2$ and $|\bar u|=4$, then $u^4$ centralises $T$ and
lies in class $A_4+A_1$. Let $L$ be a Levi subgroup of this type containing
$u^4$. It has centre $Z(L)^\circ$ of dimension~2, so this is in fact $T$. Now
$u$ normalises $T$ and hence also $[L,L]=A_4A_1$. Now $u^2$, of order~16,
acts as an inner element on this, and by Lemma~\ref{lem:innercent} and using
\cite{La95} and \cite{LS12} we arrive at a contradiction. The case
where $\dim T=4$ is similar.  \par
Finally for $G=E_8$, the same line of argument as for the other groups shows
that no new constellations occur.
\end{proof}

\begin{exmp}   \label{exmp:torus exc}
Both cases in Proposition~\ref{prop:torus exc} do actually lead to examples.
\par
(a) Let $G=E_6$ with $p=3$. By \cite[Thm.~A]{SS97}, there is a maximal subgroup
$H=D_4T_2.\fS_3$ of $G$ containing a regular unipotent element $u$, with
$T_2\unlhd H$ a 2-dimensional torus. As $u^3\ne1$ centralises $T_2$, the
subgroup $T_2\langle u\rangle$ of $H$ then lies in a proper parabolic subgroup
and so yields an example for the situation in
Proposition~\ref{prop:torus exc}(a).
\par
(b) Let $G=E_7$ with $p=2$. According to \cite[Thm.~A]{SS97} there is a maximal
subgroup $H=E_6T_1.2$ in $G$ containing a regular unipotent element $u$, with
$T_1\unlhd H$ a 1-dimensional torus. Then $T_1\langle u\rangle\le H$
yields an example for the situation in Proposition~\ref{prop:torus exc}(b).
\end{exmp}


\end{document}